\documentclass[a4paper,12pt,reqno]{article}
 \usepackage[english]{babel}
  \usepackage{amsmath,amsfonts,amssymb,amsthm,bbm}
   \usepackage{graphics,epsfig,psfrag} 
    \usepackage{authblk}
   \usepackage{paralist,subfigure}
   \usepackage[dvipsnames]{xcolor}
   \usepackage{hyperref} 

    \usepackage[active]{srcltx} 
    \usepackage{color,soul} 
    \usepackage[latin1]{inputenc}
    \usepackage[OT1]{fontenc}
    \usepackage{authblk}
    \usepackage{esint}
    \usepackage{stmaryrd}
    \usepackage{comment}

\newtheorem{theorem}{Theorem}[section]
\newtheorem{corollary}[theorem]{Corollary}
\newtheorem{lemma}[theorem]{Lemma}
\newtheorem{prop}[theorem]{Proposition}

\theoremstyle{definition}
\newtheorem{definition}[theorem]{Definition}

\def\N{\mathbb{N}}

\def\R{\mathbb{R}}

\let\e=\varepsilon
\let\r=\rho

\let\t=\tilde
\let\ol=\overline
\let\ul=\underline
\let\.=\cdot
\let\0=\emptyset

\let\mc=\mathcal

\def\1{\mathbbm{1}}

\def\ll{\llbracket}
\def\rr{\rrbracket}

\newenvironment{formula}[1]{\begin{equation}\label{#1}}
                       {\end{equation}\noindent}

\def\Fi#1{\begin{formula}{#1}}
\def\Ff{\end{formula}\noindent}

\title{\bf Propagation properties in a multi-species SIR reaction-diffusion system.}

\author[1]{Romain {\sc Ducasse}}
\author[2]{Samuel {\sc Nordmann}}
\affil[1]{Universit\'e Paris Cit\'e and Sorbonne Universit\'e, CNRS, Laboratoire Jacques-Louis Lions (LJLL), F-75006 Paris, France}
\affil[2]{Department of Applied Mathematics, Tel Aviv University}

\date{}

\begin{document}

\maketitle


\begin{abstract}
    We consider a multi-species reaction-diffusion system that arises in epidemiology to describe the spread of several strains, or variants, of a disease in a population. Our model is a natural spatial, multi-species, extension of the classical SIR model of Kermack and McKendrick. First, we study the long-time behavior of the solutions and show that there is a ``selection via propagation" phenomenon:~starting with $N$ strains, only a subset of them - that we identify - propagates and invades space, with some given speeds that we compute. Then, we obtain some qualitative properties concerning the effects of the competition between the different strains on the outcome of the epidemic. In particular, we prove that the dynamic of the model is not well characterized by the usual notion of {\em basic reproduction number}, which strongly differs from the classical case with one strain.
\end{abstract}

\section{Introduction}

\subsection{Problem under consideration}

In $1927$, Kermack and McKendrick \cite{KmcK1, KmcK2, KmcK3} introduced several deterministic models designed to describe the temporal development of a disease in a population, among which the celebrated \emph{SIR} model, which became a corner stone of mathematical epidemiology. In this model, the population under consideration is divided into three {\em compartments}: the \emph{susceptible} (who are not infected), the {\em infectious} (who have the disease and can transmit it to the susceptibles) and the {\em recovered} (who had the disease and are now immune, or dead). The original SIR model is the following system of ODEs:
\begin{equation}\label{SIR ODE}
\left\{
    \begin{array}{rlll}
    \dot S(t) &= - \alpha SI,& \quad    &t>0,  \\
   \dot I(t) &= \alpha SI - \mu I,& \quad    &t>0,\\
   \dot R(t) &= \mu I,& \quad    &t>0.
   \end{array}
\right.
\end{equation}
The functions $S(t),I(t),R(t)$ represent the number of susceptible, infectious or recovered individuals respectively, at time $t>0$. The susceptible are contaminated by the infectious, according to a law of mass-action, with a rate $\alpha  I$, where $\alpha>0$ is a constant parameter that accounts for the transmission efficiency of the disease. The infected individuals cease to be infectious with some recovery rate $\mu$ (then, $\frac{1}{\mu}$ is the mean duration of the infection). They then become recovered. The recovered individual do not become susceptible again and therefore do not play any role in the dynamic of the system.\\

We review the main results on \eqref{SIR ODE} in Section \ref{sec rap}. For now, let us mention that a serious drawback of this model is that it is purely temporal: spatial effects (migrations and localisation of individuals for instance) are not taken into account. Spatial effects are now recognised as being of first importance in the description of epidemics. Many models were introduced to bridge this gap. A very classical way to build a spatial model is to add diffusion terms on the equations. The simplest diffusive SIR model is the following (originally introduced by K\"all\'en \cite{Kallen}):
\begin{equation}\label{sir spa}
\left\{
    \begin{array}{rll}
    \partial_t S(t,x) &= - \alpha SI,\quad    &t>0, \ x\in \R, \\
   \partial_t I(t,x) &= d \Delta I + \alpha SI - \mu I, \quad &t>0, \ x\in \R, \\
   \partial_t R(t,x) &= \mu I, \quad    &t>0, \ x\in \R.
   \end{array}
\right.
\end{equation}
Now, $S,I,R$ do not only depend on the time $t$ but also on a space variable $x\in \R$. The presence of a diffusion operator on the infected individuals account for the fact that they diffuse, they move randomly, and therefore can bring the disease into places were it was not present before, potentially leading to a spatial spread of the disease. Many other spatial models were introduced, see Section \ref{sec rap} for a brief review. One may find natural to add diffusion on the $S$ and $R$ individuals also. It turns out that doing so in the equation for $R$ does not change much, but doing so in the equation for $S$ makes the analysis much more intricate, we give more details about that in Section \ref{sec rap}.\\

The goal of this paper is to answer the following natural question:
\begin{center}
\textbf{What happens when there are several strains?}
\end{center}

To answer this, we consider the following spatial and multi-species SIR model:
\begin{equation}\label{syst}
    \left\{
    \begin{array}{lll}
         \partial_t S &= -S  \left(\sum_{k=1}^N\alpha_k I_k\right), \quad &t>0, \ x\in \R,  \\
         \partial_t I_k &= d_k \Delta I_k + (\alpha_k S -\mu_k) I_k , \quad &t>0, \ x\in \R,\ k\in \llbracket 1, N \rrbracket, \\
         \partial_t R_k &= \mu_k I_k , \quad &t>0, \ x\in \R,\ k\in \llbracket 1, N \rrbracket.
    \end{array}
    \right.
\end{equation}
Here $N\in \N$ is a positive integer, and the system is therefore made of $2N+1$ equations. The parameters $\alpha_1,\ldots,\alpha_N,\mu_1,\ldots,\mu_N, d_1,\ldots,d_N$ are strictly positive constants. The unknowns are the $2N+1$ functions $(S(t,x), I_1(t,x),\ldots,I_N(t,x),R_1(t,x),\ldots,R_N(t,x))$. Being an evolution problem, \eqref{syst} has to be completed with $2N+1$ initial data.

This paper is concerned with the long-time behavior of solutions to \eqref{syst} - our main interests are the convergence to equilibrium and the speed of propagation of solutions. \\

System \eqref{syst} describes the spread of $N$ traits, or strains, or variants of a disease in a population. In the following, we will always use the generic term \emph{trait}. For $k\in\llbracket 1,N\rrbracket$, $I_k$ represents the number of individuals infected by the trait $k$. They can transmit this trait to the susceptible individuals with rate $\alpha_k$, they recover with rate $\mu_k$, and when recovering they turn into individuals of type $R_k$. The individuals $I_k$ diffuse with rate $d_k$.\\

The model \eqref{syst} is probably the simplest and most natural spatial multi-species SIR model one could come up with: the only phenomena taken into account are contamination and remission. There is no waning of immunity, there are no vital dynamics (arrival of new individuals, birth and death of individuals), and there can not be multiple infections: individuals who were infected by the trait $k$ will not be infected by any other trait. Yet, as we shall prove, this model exhibits some interesting behaviors.\\

The main result of this paper, presented in Section \ref{sec res} as Theorem \ref{main th}, is that there is a \textbf{selection via propagation} phenomenon that appears: only a given subset of all the traits (possibly empty) propagates, all the other traits vanish. Such a phenomenon can be seen as an emergent property of the model: there is no {\em a priori} selection mechanism in the equations.

In addition to finding the traits that survive, and to computing their limits, we also identify how these traits spread across space. As we will see, the system develops spatial traveling patterns such as {\em pulses} and {\em terraces}.\\

This result at hand, we use it to obtain some qualitative properties on the model. More precisely, we study how the parameters of the model (the contamination, remission, and diffusion rates, and the initial densities), influence the outcome of the epidemic. In the present paper, we consider only three properties, that illustrate how this model with several strains differs from the case with one strain.

First, we show that the so-called {\em basic reproduction number} (whose definition is recalled after) is not a good predictor of the dynamic of the model, in the sense that some strains that have large basic reproduction number may not propagate while some with low basic reproduction number would. This is somewhat counter-intuitive, as this basic reproduction number is usually considered as the criterion to compute the evolution of the disease.

The second property we investigate concerns the contamination rates of each trait that propagates. More precisely, we show that each trait that spreads is less contagious and has a lower recovery rate that all the traits that have propagated before.

Third, we study how the number of surviving individuals depends on the initial number of susceptible individuals. When there is only one strain that propagates, it is known that the number of surviving individuals is a decreasing function of the number of initially susceptible individuals. When there are several strains, we will see that this may be false.\\

Let us mention that some models with several strains were previously introduced in different contexts, sometimes in connection with evolutionary issues, we refer for instance to \cite{epid2, epid1} and the references therein. However, these models do not exhibit the same behaviors as the one we consider here: either only one of the trait propagates, or all the traits propagate - there is either no selection, or a very strong one. This is because the models usually considered in the literature take into account other effects (as internal dynamics, such as reproduction, for instance), which somehow impose a given dynamic to the system and strongly change its outcome, and also may simplify the analysis.\\

To conclude this introduction, let us mention that the system we consider, \eqref{syst}, is a reaction-diffusion system. Although there is a wide literature on such topics, there are very few results concerning general systems with more than $2$ components. A key difficulty is the lack of comparison principles for such systems. Also, the convergence to equilibrium for reaction-diffusion systems is in general a very intricate question.

Before presenting our results, we first give a quick review on the one-species SIR models \eqref{SIR ODE} and \eqref{sir spa}. This will prove useful in order to develop the intuition to understand the dynamic of \eqref{syst}.

\subsection{SIR models: review and basic results}\label{sec rap}

As we mentioned already, the original SIR model is the ODE system \eqref{SIR ODE}. The main result on this system, proved by Kermack and McKendrick in \cite{KmcK1}, is that there is a {\em threshold phenomenon}: 
\begin{itemize}
    \item If $\frac{\alpha S_0}{\mu}>1$, the disease {\em propagates}, in the sense that there is $\eta>0$ such that, no matter how small $I_0>0$ is, $\lim_{t\to+\infty}S(t) < S_0 - \eta$. 
    
    In other words, even the presence of a very small amount of infectious individuals leads to a strictly positive number of infections.

    \item If $\frac{\alpha S_0}{\mu}\leq 1$, the disease {\em fades off}, i.e., $\lim_{t\to+\infty}S(t)$ can be made as close as we want to $S_0$ by decreasing $I_0>0$.
\end{itemize}
This celebrated result led to introduce the so-called {\em basic reproduction number}
$$
\mc R_0 := \frac{\alpha S_0}{\mu},
$$
whose value dictates the behavior of the disease. This notion is now a cornerstone of mathematical epidemiology, and was extended to several other contexts.\\

As we mentioned, a natural way to turn the ODE system into a spatial system is to add diffusion terms in it. This leads to the system
\begin{equation}\label{sir diff}
\left\{
    \begin{array}{rll}
    \partial_t S(t,x) &= d_S \Delta S - \alpha SI,\quad    &t>0, \ x\in \R, \\
   \partial_t I(t,x) &= d_I \Delta I + \alpha SI - \mu I, \quad &t>0, \ x\in \R, \\
   \partial_t R(t,x) &= \mu I, \quad    &t>0, \ x\in \R.
   \end{array}
\right.
\end{equation}
We refer to \cite{Murray2} for more details on this model. Observe that we do not write diffusion on the $R$ individuals, this is because they do not really play any role in the dynamic. Adding diffusion on them would not change much the analysis and the results. The diffusion on the $S$ individuals, on the other hand, rises many complications. For instance, the boundedness of solutions is not clear at all (when $d_S = d_I$, one could sum the equations for $S$ and $I$ and apply the parabolic comparison principle, but in general, the question is rather intricate, see \cite{DucrotHeat} for instance). \\

For this reason, we consider here only the case where $d_S = 0$, that is, the susceptible do not move, and then \eqref{sir diff} boils down to \eqref{sir spa}. In this case, the threshold phenomenon of Kermack and McKendrick can be extended in the following form: consider \eqref{sir spa} and take for initial data $S(0,x) = S_0 \in \R_ +$ (i.e., the susceptible population is initially uniformly distributed), $I(0,x) = I_0(x)$ where $I_0$ is non-negative, non-zero, and compactly supported (the infectious individuals and initially localized), and take $R(0,x) = 0$. 

First, if $\frac{\alpha S_0}{\mu}\leq 1$, then the epidemic fades off in the sense that, for every $I_0$ non-negative and compactly supported, one has
$$
\lim_{\vert x \vert \to +\infty}\lim_{t\to+\infty}R(t,x) = 0\quad \text{and}\quad \lim_{\vert x \vert \to +\infty}\lim_{t\to+\infty}S(t,x) = S_0.
$$
This means that the number of recovered and susceptible individuals are unchanged after the contamination (in the limit $t\to+\infty)$, at least far away in space (in the limit $\vert x\vert \to +\infty)$: far away from the initial focus of infection (the support of $I_0$ say), no contamination occurred. Of course, in the vicinity of the initial focus of infection, even when the disease does not propagate, $\lim_{t\to+\infty}S$ and $\lim_{t\to+\infty}R$ can be arbitrarily small (resp. large), depending on the initial datum $I_0$. This is natural: even if a disease does not propagate, if half of the population of a city is initially contaminated, then at least half of the population will become recovered.

Now, if $\frac{\alpha S_0}{\mu}>1$, the disease propagates with speed $c^\star := 2\sqrt{d(\alpha S_0 - \mu)}$ and has {\em asymptotic value} $\rho>0$, in the sense that
$$
\sup_{\vert x \vert > c t} R(t,x)\underset{t \to +\infty}{\longrightarrow} 0,\quad \forall c > c^\star\quad \text{and}\quad \sup_{\delta<\vert x \vert < c t} \vert R(t,x) - \r\vert  \underset{\substack{t \to +\infty \\ \delta \to +\infty}}{\longrightarrow} 0,\quad \forall c < c^\star.
$$
This means that infections occur everywhere in space asymptotically in time, even far away from the support of $I_0$. The number $\rho$ represents the number of casualties, and one can show that $\rho$ is the unique positive solution of the transcendent equation $S_0(1-e^{-\frac{\alpha}{\mu}\rho})=\rho$. Observe that it is independent of $I_0$.

The propagation can also be seen on the evolution of the number of susceptible $S(t,x)$:
$$
\sup_{\vert x \vert > c t} \vert S(t,x) -S_0\vert\underset{t \to +\infty}{\longrightarrow} 0,\quad \forall c > c^\star\quad \text{and}\quad \sup_{\delta<\vert x \vert < c t} \vert S(t,x) - S_1\vert  \underset{\substack{t \to +\infty \\ \delta \to +\infty}}{\longrightarrow} 0,\quad \forall c < c^\star,
$$
where $S_1 := S_0 e^{-\frac{\alpha_1}{\mu_1}\rho}<S_0$.

Again, saying that the disease propagates is relevant only for large values of $\vert x \vert$, hence the $\delta$ in the above suprema. 

A crucial observation is that the dynamic far away from the support of $I_0$ will be independent of $I_0$ itself.

The above result means that, when $\frac{\alpha S_0}{\mu}>1$, then $R$ develops into two traveling waves that connect $0$ to $\rho$, one traveling to the right and one to the left, and $S$ develops into two waves that connect $S_0$ to $S_1$. The $I$ density develops into two traveling pulses (one going to the right, one going to the left). The dynamic at a given time is represented in the following drawing, where $I$ is represented with dashed line, and $S,R$ with solid lines.
\begin{center}
    \includegraphics[scale=0.7]{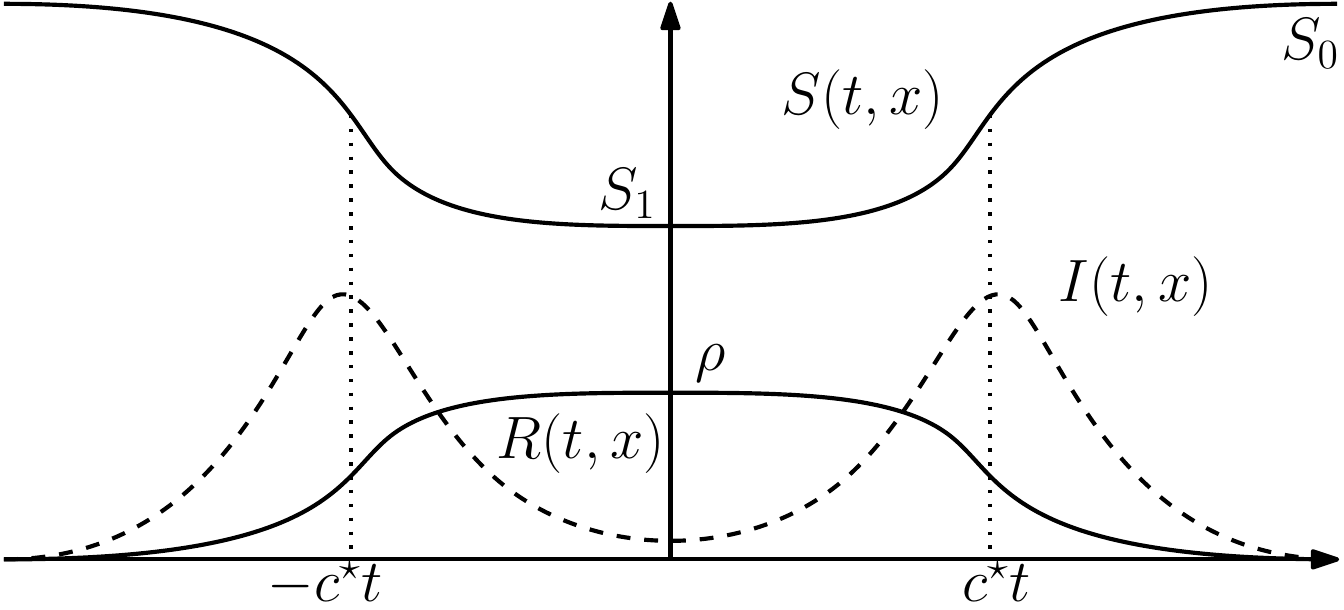}
\end{center}
The proof of this result relies on the fact that $R(t,x)$ solves
\begin{equation}\label{rd 1}
\partial_t R(t,x) = d_I \Delta R(t,x) + f(R(t,x)) +\mu I_0(x),\quad t>0,\ x\in \R,
\end{equation}
where $f(z) = \mu S_0 (1 - e^{-\frac{\alpha}{\mu}z}) - \mu z$.

This can be obtained by doing some easy algebraic manipulations on the equations (we will present these in the sequel). The resulting equation \eqref{rd 1} is a {\em KPP reaction-diffusion equation}, whose name refers to Kolmogorov, Petrovski and Piskunov, who proved some crucial results on such equations, see \cite{KPP}. The threshold phenomenon and the speed of propagation we mentioned before are then direct consequences of some now basic results from reaction-diffusion equation theory (that we will present below in the course of our proof). In particular, $R$ converges to a stationary solution of \eqref{rd 1}, that is, $R(t,x)\underset{t\to+\infty}{\longrightarrow} R_\infty(x)$, where $R_\infty$ solves
$$
d_I \Delta R_\infty(x) + f(R_\infty(x)) +\mu I_0(x) = 0,\quad  x\in \R.
$$
Moreover, one can show that $\lim_{\vert x \vert \to +\infty}R_\infty(x) = \rho$, where $\rho \in \R$ satisfies the stationary equation ``far away" in space, i.e., where $I_0$, which is compactly supported, has no influence. More precisely, $\rho$ is solution of 
$$
f(\rho)=0.
$$
It is easy to see that there is a unique positive $\rho$ solution of this equation if and only if
$f^\prime(0)=\alpha S_0 - \mu>0$.\\

For the reader interested in the case where $d_S >0$, we refer to the work of Hosono and Ilyas \cite{HI}, who proved the existence of traveling fronts for \eqref{sir diff}, and also \cite{DucrotHeat}, where the Cauchy problem is considered (under a more general form).\\

Let us briefly mention some related works. Here, we build spatial models from the ODE model by adding diffusion terms. One could also consider models where contaminations occur at distance. This was first suggested by Kendall \cite{K1, K2}, and such models were studied in the homogeneous case by Diekmann and Thieme, see \cite{Di1, Di2, T1, T2, T3, TZ}, and in the heterogeneous case by the first author, see \cite{Duc1}. Heterogeneous diffusive SIR models were also considered, see \cite{Duc2, DG}. Let us also mention the work \cite{BRRSIR} where the authors consider the influence of networks on similar models.

\subsection{Results of the paper}\label{sec res}

We now present here our main results. Before doing so, let us mention that the existence and uniqueness of classical solutions for \eqref{syst} (that is, solutions that are $C^1_tC^2_x$ for $I,R$ and $C^1_tC^0_x$ for $S$) is standard, see \cite{Henry} for instance.

We start with defining what it means for a trait to propagate, and what is the speed of propagation.

\begin{definition}[Propagation]\label{def}
Let $(S,I_1,\ldots,I_N,R_1,\ldots,R_N)$ be the solution of \eqref{syst} arising from the initial datum $(S_0,I_{1}^0,\ldots,I_N^0,0,\ldots,0)$, where $S_0$ is a positive constant and $I_{1}^0,\ldots, I_{N}^0$ are continuous, non-zero, non-negative and compactly supported.

\begin{itemize}

\item We define the {\bf asymptotic value} of the $k$-th trait, $k\in \llbracket 1, N \rrbracket$, by
$$
\r_k := \liminf_{\vert x \vert \to +\infty} \left(\lim_{t\to+\infty}R_k(t,x)\right).
$$

\item We say that the \textbf{$k$-th trait propagates} if and only if
$$
\r_k >0.
$$
If $\rho_k = 0$, we say that \textbf{the trait vanishes or goes extinct}.

\item When the $k$-th trait, $k\in \llbracket 1, N \rrbracket$,  propagates, we say that it {\bf propagates with speed} $c^\star>0$ if
$$
\left( \sup_{\vert x \vert > c t} R_k(t,x)\right) \underset{t \to +\infty}{\longrightarrow} 0,\quad \text{ for every }\ c > c^\star,
$$
and
$$
\left( \sup_{\delta<\vert x \vert < c t} \vert R_k(t,x) - \r_k\vert \right) \underset{\substack{t \to +\infty \\ \delta \to +\infty}}{\longrightarrow} 0,\quad \text{ for every }\ c < c^\star.
$$
\end{itemize}

\end{definition}
Observe that, in the system \eqref{syst}, we have that $\partial_t R_k(t,x)\geq 0$, hence, $\lim_{t\to +\infty}R_k(t,x)$ exists. Then $\rho_k$ is always well defined. 

We mention that $\lim_{t\to+\infty}R_k(t,x)$ depends on the values of the initial data $I_1^0,\ldots,I_N^0$, but $\rho_k$ does not, as we shall see. Whether a trait propagates or not will not depend on the initial data, provided they all satisfy the hypotheses of the theorem. In particular, if some initial datum is not compactly supported, the result is false. Also, if some $I_k^0$ is everywhere equal to $0$, one has to remove it and consider the problem with $N-1$ traits.\\

Let us now introduce some notations. For $N$ fixed (the number of traits that exist at initial time), for $(\alpha_k)_{\llbracket 1,N\rrbracket}, (\mu_k)_{\llbracket 1,N\rrbracket}, (d_k)_{\llbracket 1,N\rrbracket}$ given, we define some functions $c_k(\sigma), f_k(\rho,\sigma)$ and $\rho_k(\sigma)$, of variables $\rho \in \R$ and $\sigma \in \R^+$.

\begin{itemize}

\item The \emph{speed function} $c_k(\sigma)$ is defined for $\sigma \geq 0$ by

$$
c_k(\sigma) := \begin{cases}
2\sqrt{d_k(\alpha_k \sigma - \mu_k)}, \quad &\text{ when }\ \alpha_k \sigma - \mu_k\geq 0,\\
0,\quad &\text{otherwise}.
\end{cases}
$$

\item The \emph{reaction} function $f_k(\rho,\mu)$, defined for $\rho \in \R, \sigma \geq 0$, is given by
$$ 
f_k(\rho,\sigma) :=  \sigma \mu_k(1-e^{-\frac{\alpha_k}{\mu_k}\rho}) - \mu_k \rho.
$$

\item For $\sigma >0$ given, the equation $f_k(\cdot,\sigma)=0$ has a unique positive solution if and only if $\frac{\alpha_k \sigma}{\mu_k}>1$. When it exists, we denote this positive solution
$$
\rho_k(\sigma).
$$

\end{itemize}

According to the reminders we made in the previous section concerning the model~\eqref{sir spa} with one strain, we see that $c_k(\sigma), f_k(\rho,\sigma)$ and $\rho_k(\sigma)$ would be respectively the speed of propagation, the nonlinear reaction term and the asymptotic value for the $k$-th trait, provided it were the only trait present in the model and if the initial number of susceptible were $\sigma$.\\

We can now explain the heuristic underlying the results we are about to present. When there is only one trait, the result we mention in the previous section says that either it fades off, or it propagates with some speed. It it propagates, the trait consumes a certain proportion of the susceptible individuals. If we have several traits, then we can expect the competition between the traits to be negligible at first, and then each trait would start to propagate with the speed it would have provided it were alone: the trait $k$ would propagate with speed $c_k(S_0)$. The idea is that the trait that would propagate faster will indeed do so. But, doing so, it will consume a proportion of the susceptible, and therefore all the other traits will have to propagate in an environment where the number of susceptible is reduced - and the speed of propagation of each trait depends on the number of susceptible individuals.

For instance, assume that $c_1(S_0) > c_k(S_0)$ for all $k\neq 1$. Then, the trait $1$ would start to propagate, and it will leave behind a number of susceptible equal to $S_1 := S_0 e^{-\frac{\alpha_1}{\mu_1}\rho_1}$ (this is at least what happens if it were alone, as explained in the previous section). Then all the other traits will be left in an environment where the number of susceptible is equal to $S_1$, not $S_0$. Then, one can expect the above dynamic to continue this way: the trait $j\neq 1$ that maximizes $c_j(S_1)$ will propagate, and consume a proportion of the susceptible, and the dynamic should go on until either there are not enough susceptible left for the remaining traits to propagate, or all the traits have propagated.

At a given time, the situation should look like this, when $2$ traits propagate (we only represent the right hand side of space, the dynamic is similar on the left hand side):

\begin{center}
    \includegraphics[scale=0.5]{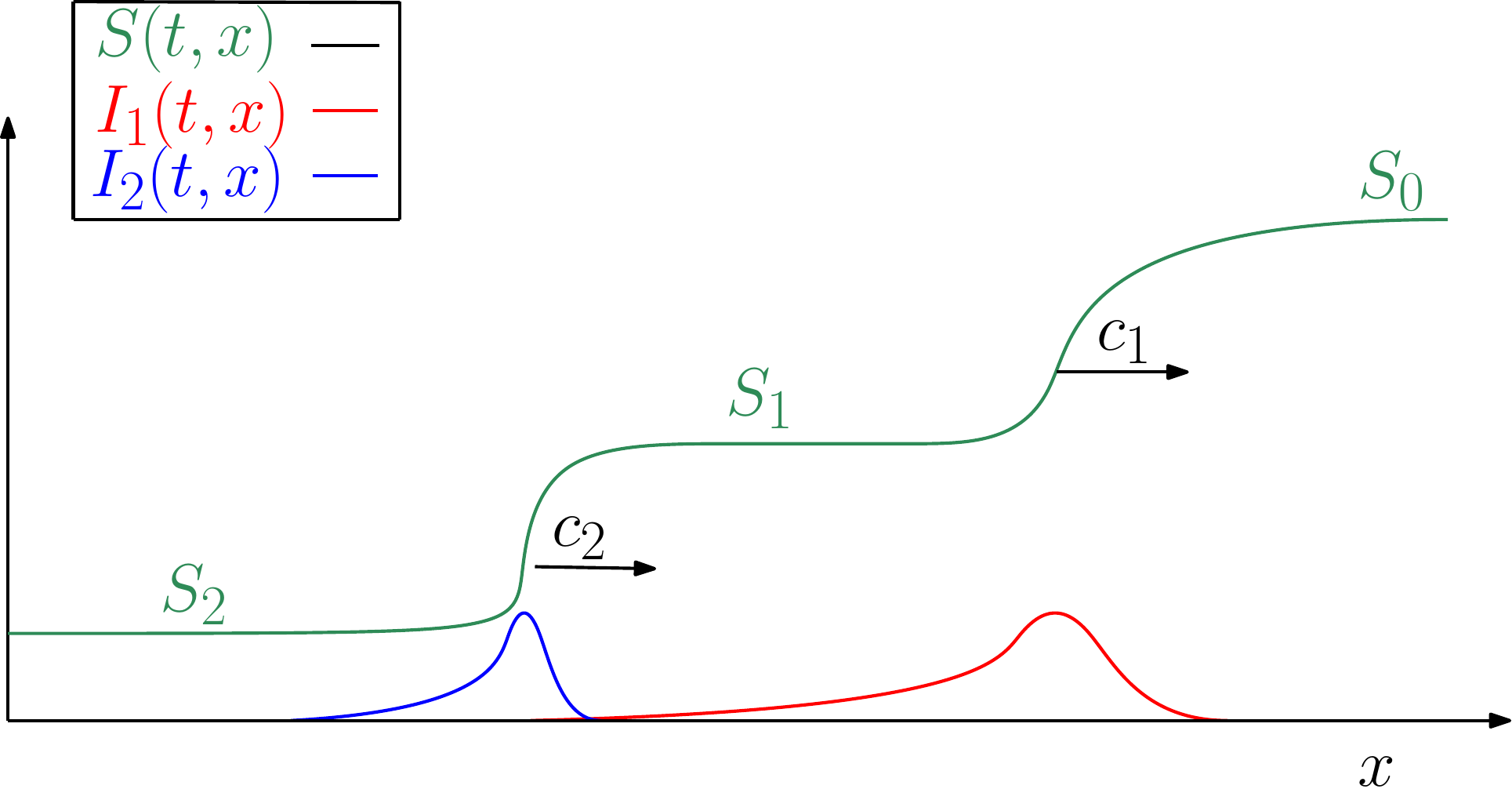}
\end{center}

We will prove that the above heuristic is indeed true, but only under some extra hypotheses. Basically, we will have to require the speeds of each trait that propagate are not ``too close". This will ensure that we control the interactions between the traits. Crucially, if this is not true, then the heuristic and the result may be false, see the remark after the statement of Theorem \ref{main th}.

To formalise the heuristic, we introduce the following:

\begin{definition}\label{def prop seq}

Let $N \in \N^\star$ and $S_0>0$ be fixed, and let $(d_k,\alpha_k,\mu_k)_{k\in\llbracket 1, N\rrbracket}$ be strictly positive real numbers. We define the \emph{propagation sequences} as four (finite) sequences $(k_i)_{i\in \llbracket 1, p\rrbracket}$, $(c_i)_{i\in \llbracket 1, p\rrbracket}$, $(\rho_i)_{i\in \llbracket 1, p\rrbracket}$ and $(S_i)_{i\in \llbracket 0, p\rrbracket}$ as follows: first, we take $S_{i=0} = S_0$, and then we define recursively:
\begin{itemize}
    \item $k_i := \text{argmax}_{ z \in \llbracket1,N\rrbracket\backslash \{k_1,\ldots,k_{i-1}\}} \left\{c_z(S_{i-1}) \ \text{such that}\ c_z(S_{i-1})>0\right\}$. For now, we assume that the argmax is well defined (that is, there is a single $z$ that maximizes $c_z(S_{i-1})$). When stating our results, we will actually need a stronger hypothesis.
    
    If the above set is empty, then $k_i$ is not defined and the sequence stops here.
    
    \item $c_i := c_{k_i}(S_{i-1})$.
    
    \item $\rho_i := \rho_{k_i}(S_{i-1})$.
    
    \item $S_i := S_{i-1}e^{-\frac{\alpha_{k_i}}{\mu_{k_i}}\rho_i}$.
\end{itemize}

\end{definition}
The heuristic we explained before then says that the traits $k_1,\dots,k_p$ are the one that should propagate, with speed $c_1,\ldots,c_p$ and each one will have asymptotic value $\rho_1,\ldots,\rho_p$ respectively.

For any trait $k_i$ that propagates, with $i = 1,\ldots,p$, then there will be $S_{i-1}$ susceptible just before this trait propagates, and $S_{i}$ susceptible after.\\

One could define the {\em final number of susceptible individuals}
$$
S_\infty := \lim_{\vert x \vert \to +\infty}\left(\lim_{t\to+\infty} S(t,x)\right).
$$
It represents the number of individuals that have survived to all the traits. When our heuristic holds true, we will have
$$
S_\infty = S_p,
$$
where $p$ is given by the propagation sequences.\\

Our main result is the following :

\begin{theorem}\label{main th}
Let $(S,I_1,\ldots, I_N,R_1,\ldots,R_N)$ be the solution of \eqref{syst} arising from the initial datum $(S_0,I_1^0,\ldots, I_N^0,0,\ldots,0)$, where $S_0\in \R^+$ and the $I_i^0$ are continuous non-negative, non-zero, compactly supported initial data, for all $i\in\llbracket 1, N \rrbracket$.

 Let $(k_i)_{i\in\llbracket 1,p\rrbracket},(c_i)_{i\in\llbracket 1,p\rrbracket},(\rho_i)_{i\in\llbracket 1,p\rrbracket},(S_i)_{i\in\llbracket 0,p\rrbracket}$ be the propagation sequences given in Definition \ref{def prop seq}.

Assume that, for all $i\in \llbracket 1,p-1\rrbracket$,
\begin{equation}\label{hyp}
c_{k}(S_{i-1}) + c_k(S_i) < c_{i},\quad \text{for all} \ k\neq k_1,\ldots,k_i, 
\end{equation}
and that
\begin{equation}\label{hyp 2}
\alpha_k S_p - \mu_k<0,\quad \text{for all} \ k\neq k_1,\ldots,k_p.  
\end{equation}
Then, the traits $k_1,\ldots,k_p$ propagate in the sense of Definition \ref{def}, with speeds $c_1,\ldots,c_p$, and have asymptotic values $\mc \rho_1,\ldots,\rho_p$. All the other traits do not propagate and fade off.
\end{theorem}
Let us give some remarks on this theorem. First, if, when computing the propagation sequences, one finds that $p=0$ or $p=1$, then the hypothesis \eqref{hyp} is void, and the theorem only requires that \eqref{hyp 2} to hold true. Similarly, if one finds that $p=N$, then the hypothesis \eqref{hyp 2} is void and only \eqref{hyp} is required to apply the theorem.

When $N=1$ and $p=0$, observe that our theorem does not completely boils down to the result on the model with a single trait recalled in Section \ref{sec rap}. Indeed, in this case, we need \eqref{hyp 2} to be able to say that the (single) trait does not propagate, but as we mentioned in Section \ref{sec rap} it is sufficient to have $\alpha_1 S_0 - \mu_1 \leq 0$.

Now, as we can see, Theorem \ref{main th} proves that our heuristic is true provided that \eqref{hyp} and \eqref{hyp 2} hold true. These restrictions are actually necessary in the sense that the result may be false without them. Indeed, hypothesis \eqref{hyp} says that when a trait propagates, the speed of the other traits that could propagate are ``sufficiently" smaller. If two traits have almost equal speeds, then the two fronts that should arise can actually interfere, and the slower front could accelerate. This phenomenon was observed for instance in the papers \cite{Lam,Lam2}, in the case of competition models. Actually, in the case where there are two traits (i.e., $N=2$ in \eqref{syst}), then we believe that one could use the results presented in these papers to get rid of \eqref{hyp} and \eqref{hyp 2} in Theorem \ref{main th} (but then the traits that propagate are not given bu our propagation sequence). When $N>2$, the situation seems more intricate, and we leave it for further study.\\

Once that we have Theorem \ref{main th}, we can use it to obtain some qualitative properties of our model. Indeed, this theorem tells us how the solutions behave according to the values of the parameters: if we want to know which traits propagate, or how many individuals are contaminated or left untouched after the epidemic, then we can compute the propagation sequences. However, these can be rather intricate to compute. We do not plan here to give a complete picture of the qualitative properties of the model, we leave it for further study, but we will nevertheless show three properties who shed light on some aspects of the system \eqref{syst}.

First, we shall see that, when several strains can propagate, then the basic reproduction number of each strain does not indicate which one can propagate.

More precisely, we shall construct examples of situations with two traits, one which has a large basic reproduction number, and one with a small reproduction number, and where the trait with the small basic reproduction number propagates, while the other fades off. This is the object of Corollary \ref{cor 1} below.

Second, we prove that, when several traits propagate, each one is less contagious (i.e., it has a lower contamination rate $\alpha_i$) and is less ``mortal" in the sense that its remission rate $\mu_i$ is lesser, than each trait that propagates before. This result is the Corollary \ref{cor 2} below.

Third, we study how the number of individuals left untouched after the epidemic depends on the number of susceptible individuals initially present. We will consider the case with two traits. The picture here can be rather different from the case with a single trait. This is the object of Corollary \ref{cor 3}.

We give the precise statements and the proofs of this three properties, which are corollaries of Theorem \ref{main th}, in Section \ref{sec quali}. First, we dedicate the next Section \ref{sec proof}, to the proof of Theorem \ref{main th}.

\section{Proof of Theorem \ref{main th}}\label{sec proof}

\subsection{Outline of the proof}

This section is dedicated to the proof of our main result, Theorem \ref{main th}. In this whole section, the parameters $S_0,(\alpha_i)_{i\in\llbracket 1, N \rrbracket},(\mu_i)_{i\in\llbracket 1, N \rrbracket},(d_i)_{i\in\llbracket 1, N \rrbracket}$ are fixed, strictly positive constants, and $(k_i)_{i\in\llbracket 1,p\rrbracket}$, $(c_i)_{i\in\llbracket 1,p\rrbracket}$, $(\rho_i)_{i\in\llbracket 1,p\rrbracket}$, $(S_i)_{i\in\llbracket 0,p\rrbracket}$ denote the propagation sequences given in Definition \ref{def prop seq}. We also assume that the initial data for the $(I_i)_{i\in \llbracket 1,N\rrbracket}$, denoted $(I^0_i)_{i\in \llbracket 1,N\rrbracket}$, are continuous, non-zero, non-negative and compactly supported.

We prove Theorem \ref{main th} by (finite) induction. For $j \in \llbracket 0, p \rrbracket$, we denote $(H_j)$ the induction hypothesis:
 \begin{equation}\label{Hj}
     \text{ The traits }\ k_1,\ldots,k_j \ \text{propagate with speeds}\ c_1,\ldots, c_j \ \text{and asymptotic values}\ \r_1,\ldots,\r_j. \tag{$H_j$}
 \end{equation}

When $j=0$, we use the convention that the hypothesis $(H_0)$ is empty (that is, when we write: assume \eqref{Hj}, when $j=0$, we have nothing to assume).\\

To prove Theorem \ref{main th}, we will prove first that $(H_p)$ holds true, and then that every trait $k\neq k_1,\ldots,k_p$ does not propagate. A key point in the proof will be to obtain suitable exponential estimates on the functions $I_k,R_k$, for all $k$. This is the object of the following:

\begin{prop}\label{prop exp}
Let $(S,I_1,\ldots, I_N,R_1,\ldots,R_N)$ be the solution of \eqref{syst} arising from the initial datum $(S_0,I_1^0,\ldots, I_N^0,0,\ldots,0)$.

Assume that the separation of speeds hypothesis \eqref{hyp} holds true and that the induction hypothesis \eqref{Hj} is verified for some $j\in \llbracket 0,p\rrbracket$.

Then, for all $k \neq k_1, \ldots,k_j$, for all $\e>0$ small enough, there are $A^\e_{k,j},B^\e_{k,j}>0$ such that
$$
I_k(t,x) \leq A^\e_{k,j} e^{-\frac{(c_k(S_j)+\e)}{2d_k}\big(\vert x \vert - (c_k(S_j)+\e)t\big)},\quad \text{for all } \ t>0,\ x \in \R,
$$
and
$$
R_k(t,x) \leq B^\e_{k,j} e^{-\frac{(c_k(S_j)+\e)}{2d_k}\big(\vert x \vert - (c_k(S_j)+\e)t\big)},\quad \text{for all } \ t>0,\ x \in \R.
$$
\end{prop}
This proposition will allow us to obtain upper bounds on the speed of propagation. 

Observe that the estimates on the functions $R_k$ are a consequence of the estimates on the $I_k$. Indeed, we have $\partial_t R_k = \mu_k I_k$, integrating the estimate on $I_k$ yields the same estimate on $R_k$ with $B^\e_{k,j} = \mu_k A^\e_{k,j}\frac{2 d_k}{(c_k(S_k)+\e)^2} $.

We repeat that the hypothesis \eqref{hyp} is indeed necessary to have this result: it would be false otherwise.

Let us mention that, in the above Proposition \ref{prop exp}, it is possible to have $c_k(S_j)=0$ for some $k,j$. In this case, the exponential in the proposition travels with a speed $\e>0$ as small as we want.

 The proof of Proposition \ref{prop exp} is done in several steps, each one being a lemma. First, we will rewrite the system \eqref{syst} under a different form, more precisely, we will write it as a system implying only the functions $R_k$, $k\in\llbracket 1, N \rrbracket$. This is the object of Lemma \ref{lem rw}.

 Then, we will see that this lemma implies that the functions $I_k,R_k$ are respectively subsolutions and supersolutions of some equations, this is done in Lemma \ref{lem ineg}. Combining these two lemmas with another technical result (Lemma \ref{lem tech}), we will finally get Proposition \ref{prop exp}.

 We then use this proposition to show that, if $(H_j)$ is verified for some $j\in \llbracket 0,p-1\rrbracket$, then $(H_{j+1})$ also holds true. This is the object of Proposition \ref{prop under} below. Eventually, this yields that $(H_p)$ holds true.

 Finally, we prove in Proposition \ref{prop no prop} that the traits $k\neq k_1,\ldots,k_p$ do not propagate, and this will conclude the proof of Theorem \ref{main th}.

 \subsection{Proof of the exponential estimates, Proposition \ref{prop exp}}
 
We start with rewriting the system \eqref{syst}. The next lemma strongly relies on the fact that there is no diffusion on the susceptible individuals.

\begin{lemma}\label{lem rw}
Let $(S,I_1,\ldots, I_N,R_1,\ldots,R_N)$ be the solution of \eqref{syst} arising from the initial datum $(S_0,I_1^0,\ldots, I_N^0,0,\ldots,0)$.

\begin{itemize}
    \item The function $S(t,x)$ can be written as
    $$
    S(t,x) = S_0 \prod_{i=1}^N e^{-\frac{\alpha_i}{\mu_i}R_i(t,x)},\quad \forall t>0, \ x\in \R.
    $$
    
    \item For any $k\in \llbracket 1,N\rrbracket$, for $t>0, x\in \R$, we have
    $$
    \partial_t R_k = d_k \Delta R_k + \mu_kI_k^0-\mu_kR_k + \alpha_kS_0 \int_0^t e^{-\frac{\alpha_k}{\mu_k}R_k}\partial_t R_k \left(\prod_{i\neq k}e^{-\frac{\alpha_i}{\mu_i}R_i}\right)d\tau.
    $$
\end{itemize}
\end{lemma}

\begin{proof}
The proof follows from the simple observation that the first equation in \eqref{syst}, giving $S(t,x)$, can be rewritten, using the equations for $R_1,\ldots,R_N$, as
$$
\frac{\partial_t S}{S} = -\sum_{i=1}^N \frac{\alpha_i}{\mu_i} \partial_t R_i.
$$
Integrating this relation directly gives the result.

To get the second point, we simply replace $S(t,x)$ by the expression we just found in the equation giving $I_k$, thus obtaining
$$
\partial_t I_k = d_k \Delta I_k +\left(\alpha_k S_0\prod_{i=1}^Ne^{-\frac{\alpha_i}{\mu_i}R_i} - \mu_k\right) I_k.
$$
Then, we multiply by $\mu_k$ and we integrate over the $t$ variable, from $0$ to some $t>0$. Remembering that $\partial_tR_k =\mu_k I_k$, we find the result.
\end{proof}

The next lemma builds on the previous one, and shows that the functions $I_k,R_k$, $k\in \llbracket 1,N\rrbracket$, are respectively subsolutions/supersolutions of some parabolic equations on some domains.
\begin{lemma}\label{lem ineg}
Let $(S,I_1,\ldots, I_N,R_1,\ldots,R_N)$ be the solution of \eqref{syst} arising from the initial datum $(S_0,I_1^0,\ldots, I_N^0,0,\ldots,0)$.\\

Assume that the induction hypothesis \eqref{Hj} holds true for some $j\in \llbracket 0,p\rrbracket$. Then, the following inequalities hold true:
\begin{itemize}
    \item For every $\theta >1$, $\e>0$, there is $T>0$ such that
    $$
    \partial_t I_k \leq d_k \Delta I_k + (\alpha_k S_{j} \theta - \mu_k)I_k,\quad\text{for}\ t>T, \ \vert x \vert < (c_j-\e)t, \ k\in \llbracket 1 , N \rrbracket.
    $$

    \item For $k\neq k_1,\ldots,k_j$, for any $\omega \in (0,1)$, there is $L>0$ such that
$$
    \partial_t R_k \geq d_k \Delta R_k +f_k(R_k,\omega S_j) - \mu_k  S_j \left(\sum_{i\neq k, k_1,\ldots, k_j} \frac{\alpha_i}{\mu_i}R_i \right)R_k,\quad \text{for }\ t>0, \vert x\vert >L.
$$
\end{itemize}
\end{lemma}
When $j=0$, the first point in the above lemma is not well-written ($c_j$ is not defined when $j=0$). In this case, we use the convention that the first point should read: for every $\theta >1$, $\e>0$, there is $T>0$ such that
    $$
    \partial_t I_k \leq d_k \Delta I_k + (\alpha_k S_{0} \theta - \mu_k)I_k,\quad\text{for}\ t>T, \  \ k\in \llbracket 1 , N \rrbracket,
    $$
    that is, there is no restriction on $x$.
    
    The second point reads similarly when $j=0$, the sum in the right hand term simply runs on all $i\in \llbracket 1, N\rrbracket\backslash k$.

\begin{proof}
Let $j\in \llbracket 0,p\rrbracket$ be fixed and assume that \eqref{Hj} holds true. We denote $K_j := \{k_1,\ldots,k_j\}$ and $K_j^c := \llbracket 1,N\rrbracket \backslash K_j$ (if $j=0$, then $K_j = \emptyset$).\\

\emph{First point.}  We already know from Lemma \ref{lem rw} that, for all $k\in\ll 1,N\rr$,
$$
\partial_t I_k = d_k \Delta I_k + \alpha_k S_0 \left(\left(\prod_{i\in K_j} e^{-\frac{\alpha_i}{\mu_i}R_i}\right)\left( \prod_{i\in K_j^c} e^{-\frac{\alpha_i}{\mu_i}R_i}\right) - \mu_k\right) I_k.
$$
Owing to \eqref{Hj}, the traits $i\in K_j$ all propagate, with speed at least $c_j$ (because $c_1>c_2>\ldots >c_j$), hence, for every $\eta,\e>0$, there is $T>0$ such that
$$
R_i(t,x) \geq \r_i - \eta,\quad \text{ for all }\ i\in K_j, \ t>T,\ \vert x \vert < (c_j -\e)t.
$$
Then, for $t>T, \vert x \vert< (c_j-\e)t$, we have
$$
\prod_{i\in K_j} e^{-\frac{\alpha_i}{\mu_i}R_i(t,x)} \leq \prod_{i\in K_j} e^{-\frac{\alpha_i}{\mu_i}\r_i}  \prod_{i\in K_j} e^{\frac{\alpha_i}{\mu_i}\eta} 
$$
Take $\eta>0$ small enough so that $\prod_{i\in K_j} e^{\frac{\alpha_i}{\mu_i}\eta} \leq \theta$. On the other hand, we bound from above $\prod_{i\in K_j^c} e^{-\frac{\alpha_i}{\mu_i}R_i(t,x)}$ by $1$ and we recall that, by definition, $S_j = S_0 \prod_{i\in K_j} e^{-\frac{\alpha_i}{\mu_i}\r_i}$, to get 
$$
\partial_t I_k \leq d_k \Delta I_k + (\alpha_k  S_j \theta - \mu_k) I_k,\quad \text{for }\ t>T, \ \vert x \vert < (c_j-\e)t.
$$

\medskip
\emph{Second point.}
Let us take $k\neq k_1,\ldots,k_j$. Owing to Lemma \ref{lem rw}, we have
$$
    \partial_t R_k = d_k \Delta R_k + \mu_kI_k^0-\mu_kR_k + \alpha_kS_0 \int_0^t e^{-\frac{\alpha_k}{\mu_k}R_k}\partial_t R_k \left(\prod_{i\neq k}e^{-\frac{\alpha_i}{\mu_i}R_i}\right)d\tau.
$$
Because all the $R_i$ are non-decreasing with respect to $t$, we find that
\begin{align*}
    \partial_t R_k &\geq d_k \Delta R_k + \mu_kI_k^0-\mu_kR_k + \alpha_kS_0 \left(\prod_{i\neq k}e^{-\frac{\alpha_i}{\mu_i}R_i(t,x)}\right)\int_0^t e^{-\frac{\alpha_k}{\mu_k}R_k}\partial_t R_k d\tau \\
    &\geq d_k \Delta R_k -\mu_kR_k + \mu_kS_0 \left(\prod_{i\neq k}e^{-\frac{\alpha_i}{\mu_i}R_i(t,x)}\right)(1-e^{-\frac{\alpha_k}{\mu_k}R_k}).
\end{align*}
Now, owing to \eqref{Hj}, for $i\in K_j$, we know that $R_i$ propagates and has asymptotic value $\r_i$. Therefore (because $R_i$ is time-nondecreasing), for every $\eta>0$, there is $L>0$ such that
$$
R_i(t,x) \leq \r_i + \eta,\quad \text{for}\ t>0,\ \vert x \vert >L, \ i\in K_j.
$$
This allows us to derive a lower bound on $\prod_{i\neq k}e^{-\frac{\alpha_i}{\mu_i}R_i(t,x)}$ for $t>0$ and $\vert x \vert >L$ (recall that $k\not\in K_j$):
\begin{align*}
\prod_{i\neq k}e^{-\frac{\alpha_i}{\mu_i}R_i(t,x)} &=\prod_{i\in K_j}e^{-\frac{\alpha_i}{\mu_i}R_i(t,x)}\prod_{i\in K_j^c \backslash k}e^{-\frac{\alpha_i}{\mu_i}R_i(t,x)}\\
&\geq \left(\prod_{i\in K_j}e^{-\frac{\alpha_i}{\mu_i}\r_i}\right)\left(\prod_{i\in K_j}e^{-\frac{\alpha_i}{\mu_i}\eta}\right)\left(1-\sum_{i\in K_j^c \backslash k}\frac{\alpha_i}{\mu_i}R_i(t,x)\right).
\end{align*}
Hence, up to taking $\eta$ small enough so that $\prod_{i\in K_j}e^{-\frac{\alpha_i}{\mu_i}\eta}>\omega$, we get, for $t>0$, $\vert x \vert >L$,
\begin{align*}
    \partial_t R_k 
    &\geq d_k \Delta R_k -\mu_kR_k + \mu_kS_0 \omega\prod_{i\in K_j}e^{-\frac{\alpha_i}{\mu_i}\r_i}\left(1-\sum_{i\in K_j^c \backslash k}\frac{\alpha_i}{\mu_i}R_i(t,x)\right)\left(1-e^{-\frac{\alpha_k}{\mu_k}R_k}\right)\\
    &\geq d_k \Delta R_k -\mu_kR_k + \mu_k \omega S_j\left(1-\sum_{i\in K_j^c \backslash k}\frac{\alpha_i}{\mu_i}R_i(t,x)\right)\left(1-e^{-\frac{\alpha_k}{\mu_k}R_k}\right)\\
    &\geq d_k \Delta R_k +f_k(R_k,\omega S_j)  - \mu_k S_j\omega \left(\sum_{i\in K_j^c \backslash k}\frac{\alpha_i}{\mu_i}R_i(t,x)\right)\left(1-e^{-\frac{\alpha_k}{\mu_k}R_k}\right)\\
        &\geq d_k \Delta R_k +f_k(R_k,\omega S_j)  - \alpha_k S_j\left(\sum_{i\in K_j^c \backslash k}\frac{\alpha_i}{\mu_i}R_i(t,x)\right)R_k,
\end{align*}
hence the result.
\end{proof}

Now, in order to prove Proposition \ref{prop exp}, we need the following technical lemma:

\begin{lemma}\label{lem tech}
Let $d,r>0$ be chosen, and let $u(t,x)\geq 0$ be $C^1_tC^2_x$, bounded, such that:
\begin{itemize}
    \item There are $c_1, A>0$ such that, for all $t>0$, $x\in \R$, we have
    $$
u(t,x) \leq A e^{-\frac{c_1}{2d} (x -c_1 t) }.
$$
\item There are $T, c_2>0$ such that, for all $t>T,  \vert x \vert < c_2t$, we have
$$
\partial_t u -d \Delta u \leq r u.
$$
\end{itemize}
In addition, define
$$
c_3 := 2\sqrt{r d}.
$$
Assume that
\begin{equation}\label{sep speed}
    c_1 > c_3 \ \text{and} \ c_2 > c_1 + c_3.
\end{equation}
Then, there is $B>0$ such that, for $t>0$, $x\in \R$,
$$
u(t,x) \leq B e^{- \frac{c_3}{2d} (x - c_3 t)}.
$$
\end{lemma}

Observe that in this lemma, we need an hypothesis on $c_1,c_2,c_3$, namely \eqref{sep speed}. When applying this lemma later, this requirement will force us to assume the separation hypothesis \eqref{hyp} on the speeds of each trait.

\begin{proof}
For $B>\max\{ A, \sup_{t,x} u\}$, to be chosen large enough after, define
$$
w(t,x) := B e^{-\frac{c_3}{2 d} (x - c_3 t)}.
$$
Observe first that, for $t>0,  x  > c_2t$, we indeed have $A e^{-\frac{c_1}{2d}(x-c_1 t)}\leq w(t,x)$, and then $w\geq u$ there. For $t>0, x < -c_2 t$, we have $w(t,x) \geq B$, and then $w\geq u$ also in this region. Up to taking $B$ large enough, we can also ensure that $w\geq u$ in the bounded domain $\{ (t,x) \ : t\in (0,T),\ \vert x \vert < c_2t \}$.

It remains to prove that $w\geq u$ on the domain  $E := \{ (t,x) \ : t>T,\ \vert x \vert < c_2t \}$. To do so, we use the parabolic comparison principle (see \cite{PW}).

We have
$$
\partial_t w - d \Delta w = rw,\quad \text{ for } t>0, x\in \R.
$$
By hypothesis, we know that $u$ is subsolution of this equation on the domain $E$.

To apply the parabolic comparison principle, we need that $w \geq u$ on $\partial E$, that is, we need to have
$w(T,x)\geq u(T,x)$ for $\vert x \vert \leq c_2T$ and $w(t,x)\geq u(t,x)$ for $t>T$ and $\vert x \vert = c_2 t$.

Because $u$ is bounded, taking $B$ large enough guarantees that the first requirement is verified. For the second requirement, it is sufficient to have
$$
 B e^{-\frac{c_3}{2d} (c_2 - c_3 )t}\geq A e^{-\frac{c_1}{2d} (c_2-c_1)t}, \quad \text{for}\ t>T,
$$
that is, up to increasing $B$, this is true if we have $\frac{c_3}{2d} (c_2 - c_3 )< \frac{c_1}{2d} (c_2-c_1)$, but this condition is guaranteed by \eqref{sep speed}. Then, we can indeed apply the parabolic comparison theorem to $u,w$ on $E$ to get the result.
\end{proof}

We can now finally turn to the proof of Proposition \ref{prop exp}.

\begin{proof}[Proof of Proposition \ref{prop exp}]
We will prove that, under the hypotheses of Proposition \ref{prop exp}, we have, for $\e>0$ small enough,
$$
I_k(t,x) \leq A^\e_{k,j} e^{-\frac{(c_k(S_j)+\e)}{2d_k}( x  - (c_k(S_j)+\e)t)},\quad \text{for all } \ t>0,\ x \in \R.
$$
This is not exactly was is required, because we do not have the absolute value on the $x$ in the exponential. To get it, one can either redo the same arguments with supersolution moving toward the left (that is, supersolutions of the form $e^{\lambda(x+ct)}$), or, equivalently, one can do a change of variable $x\to -x$, that does not change the equations in \eqref{syst} and apply the results to these new equations.

We argue by induction. In this proof, the induction hypothesis is denoted $(P_j)$, for $j\in \llbracket 0,p\rrbracket$ and is the following:
\begin{multline}\label{Pj}
(H_j) \implies  \forall k\neq k_1,\ldots,k_j,\ \forall \e >0 \ \text{small enough},\ \exists A^\e_{k,j}>0 \ \text{such that}\\
I_k(t,x) \leq A^\e_{k,j} e^{-\frac{(c_k(S_j)+\e)}{2d_k}(x - (c_k( S_j)+\e)t)},\quad \text{for all } \ t>0,\ x \in \R.\tag{$P_j$}
\end{multline}

When $j=0$, by convention $(H_0)$ is the void hypothesis, it is always true. Therefore, $(P_0)$ would hold true if we can prove the upper bounds on the $I_k$ functions, for all $k\in\llbracket 1,N\rrbracket$, unconditionally.

To do so, start with observing that $S(t,x)\leq S_0$. This implies that, for every $k\in  \llbracket 1 , N \rrbracket$, $I_k$ satisfies
$$
\partial_t I_k \leq d_k\Delta I_k +(\alpha_k S_0 - \mu_k)I_k,\quad t>0,\ x\in \R.
$$
For any $\e>0$, the function
$$
w(t,x) = Ae^{-\frac{(c_k( S_0)+\e)}{2d_k}(x-(c_k(S_0)+\e)t)}
$$
is supersolution of this equation. Indeed, $\partial_t w - d_k \Delta w = \frac{(c_k(S_0)+\e)^2}{4d_k}w  \geq (\alpha_k S_0 - \mu_k)w$.
Hence, provided $A$ is large enough so that $w(0,\cdot)\geq I_k^0$, the parabolic comparison principle yields the result.\\

Now, assume the property \eqref{Pj} is verified for some $j\in\llbracket 0, p-1\rrbracket$, and let us show that $(P_{j+1})$ is true. 

To do so, we assume that $(H_{j+1})$ holds true. Therefore, $(H_j)$ is verified, and then $(P_j)$ tells us that
for all $ k \neq k_1, \ldots,k_j$, for $\e>0$ small enough, there is $A^\e_{k,j}>0$ such that
$$
I_k(t,x) \leq A^\e_{k,j} e^{-\frac{(c_k( S_j)+\e)}{2d_k}(x - (c_k(S_j)+\e)t)},\quad \text{for all } \ t>0,\ x \in \R.
$$

In addition, because $(H_{j+1})$ is assumed to hold, we can apply Lemma \ref{lem ineg}, to get that, for $\theta >1$, $\eta>0$, there is $T>0$ such that
    $$
    \partial_t I_k \leq d_k \Delta I_k + (\alpha_k  S_{j+1} \theta - \mu_k)I_k,\quad t>T, \ \vert x \vert < (c_{j+1}-\eta)t, \ k\in \llbracket 1 , N \rrbracket.
    $$

Define $r:= \frac{(c_k(S_{j+1})+\e)^2}{4d_k}$. Then, up to taking $\theta$ close enough to $1$, we can ensure that $\alpha_k  S_{j+1} \theta - \mu_k \leq r$, and then
    $$
    \partial_t I_k \leq d_k \Delta I_k + r I_k,\quad t>T, \ \vert x \vert < (c_{j+1}-\eta)t, \ k\in \llbracket 1 , N \rrbracket.
    $$

Now, define
$$
c_1 = c_k(S_j)+\e, c_2 = c_{j+1}-\eta, c_3 = c_k(S_{j+1})+\e.
$$
Then, up to taking $\e,\eta>0$ small enough, we can ensure that
$$
c_1+c_3  <c_2, 
$$
thanks to the hypothesis \eqref{hyp}. We always have $c_1>c_3$, because $S_{j+1}<S_j$. Then, we can apply Lemma \ref{lem tech} and this yields the result.
\end{proof}

\subsection{Conclusion. Proof of Theorem \ref{main th}}

We now use Proposition \ref{prop exp} to prove that the hypothesis $(H_p)$ is true, that is, the traits $k_1,\ldots,k_p$ computed by the propagation sequences indeed propagate, with their prescribed speeds to their prescribed asymptotic values. To prove this, we argue by induction again, this is the object of the following:
\begin{prop}\label{prop under}
Let $(S,I_1,\ldots, I_N,R_1,\ldots,R_N)$ be the solution of \eqref{syst} arising from the initial datum $(S_0,I_1^0,\ldots, I_N^0,0,\ldots,0)$.

Assume that \eqref{Hj} holds true for some $j\in\llbracket 0,p-1\rrbracket$. Then the trait $k_{j+1}$ propagates with speed $c_{j+1}$ and has asymptotic value $\r_{j+1}$, that is, $(H_{j+1})$ also holds true.
\end{prop}
We start with the following lemma, that uses Proposition \ref{prop exp} to improve Lemma~\ref{lem ineg}.
\begin{lemma}\label{lem sub}
Let $(S,I_1,\ldots, I_N,R_1,\ldots,R_N)$ be the solution of \eqref{syst} arising from the initial datum $(S_0,I_1^0,\ldots, I_N^0,0,\ldots,0)$.

\begin{itemize}
    \item Assume that \eqref{Hj} holds true for some $j\in\llbracket 0,p-1\rrbracket$. Let $\t c = \max_{i\neq k_1,\ldots,k_{j+1}} c_i(S_j) $. For any $\eta>0$, $\omega\in(0,1)$ and for every $c>\tilde c$, there is $T>0$ such that, for  $t>T$ and $\vert x \vert >c t$,
\begin{equation}\label{eq mf}
 \partial_t R_{k_{j+1}} \geq d_{k_{j+1}} \Delta R_{k_{j+1}} +f_{k_{j+1}}(R_{k_{j+1}},\omega S_j) - \eta R_{k_{j+1}}.
\end{equation}
    
    \item Assume that \eqref{Hj} holds true for some $j\in\llbracket 0,p\rrbracket$. For any $\theta>1$, for any $k\neq k_1,\ldots,k_j$, for any $c\in (0,c_j)$, there are $C,q>0$ such that, for $\vert x \vert <c t$,
    
    \begin{equation}\label{R supersol}
    \partial_t R_k \leq d_k \Delta R_k + \mu_kI_k^0 + C e^{-q\vert x \vert} + f_{k}(R_k,\theta S_j).
  \end{equation}
\end{itemize}
\end{lemma}

\begin{proof}

\medskip
\emph{First point. $R_{k_{j+1}}$ is supersolution of a reaction-diffusion equation.}\\

Let $j\in\llbracket 0,p-1\rrbracket$ be fixed and assume that \eqref{Hj} holds true. Owing to Lemma \ref{lem ineg}, we know that, for any $\omega\in (0,1)$, there is $L>0$ such that $R_{k_{j+1}}$ satisfies, for $t>0$, $\vert x\vert >L$,
$$
    \partial_t R_{k_{j+1}} \geq d_{k_{j+1}} \Delta R_{k_{j+1}} +f_{k_{j+1}}(R_{k_{j+1}},\omega S_j) - \mu_{k_{j+1}} S_j \left(\sum_{i\neq k_1,\ldots, k_j,k_{j+1}} \frac{\alpha_i}{\mu_i}R_i \right)R_{k_{j+1}}.
$$
For $i\neq k_1,\ldots,k_j$, for $\e>0$ small enough, the bounds provided by Proposition \ref{prop exp} tell us that $R_i(t,x) \leq B^\e_{i,j} e^{-\frac{(c_i(S_{j})+\e)}{2d_i}(\vert x\vert - (c_{i}(S_j)+\e)t)} $, for some $B^\e_{i,j} >0$. Hence, for $t>0$, $\vert x \vert >L $,
\begin{multline*}
    \partial_t R_{k_{j+1}} \geq d_{k_{j+1}} \Delta R_{k_{j+1}} +f_{k_{j+1}}(R_{k_{j+1}},\omega S_j) - \\
    \mu_{k_{j+1}} S_j \left(\sum_{i\neq k_1,\ldots, k_j,k_{j+1}} \frac{\alpha_i}{\mu_i}B^\e_{i,j} e^{-\frac{(c_i(S_{j})+\e)}{2d_i}(\vert x\vert - (c_{i}(S_j)+\e)t)} \right)R_{k_{j+1}}.
\end{multline*}

Let $c> \max_{i\neq k_1,\ldots,k_{j+1}} c_i( S_j) = \tilde c$. Up to taking $\e$ smaller if needed, we ensure that $\t c +\e< c$. Hence, for $\vert x\vert >c t$, for all $i\neq k_1,\ldots,k_{j+1}$, we have
$$
e^{-\frac{(c_i(S_{j})+\e)}{2d_i}(\vert x\vert - (c_{i}(S_j)+\e)t)}\leq e^{-\frac{(c_i( S_{j})+\e)}{ 2d_i}(c-\t c -\e) t}.
$$
Therefore, for any $\eta>0$, there is $T>0$ large enough so that
$$
 \partial_t R_{k_{j+1}} \geq d_{k_{j+1}} \Delta R_{k_{j+1}} +f_{k_{j+1}}(R_{k_{j+1}},\omega S_j) - \eta R_{k_{j+1}},\quad \text{for}\  t>T,\ c t<\vert x\vert .
$$

\medskip
\emph{Second point. $R_{k}$ is subsolution of a reaction-diffusion equation.}\\

Let $j\in\llbracket 0,p\rrbracket$ be fixed and assume that \eqref{Hj} holds true. Let $k\not\in K_j = \{k_1,\ldots,k_j\}$ be chosen. Owing to Lemma \ref{lem rw}, we know that $R_k$ satisfies
    $$
    \partial_t R_k \leq d_k \Delta R_k + \mu_kI_k^0 -\mu_kR_k + \alpha_kS_0 \int_0^t e^{-\frac{\alpha_k}{\mu_k}R_{k}}\partial_t R_k \left(\prod_{i\in K_j}e^{-\frac{\alpha_i}{\mu_i}R_i}\right)d\tau
    $$
    Let $c<c_j$ be chosen. We take $\ol c$ such that $\ol c \in (\max\{c, c_{j+1}\},c_j)$ (with the convention that, if $j=p$, then $c_{p+1}:=0$).

    The hypothesis \eqref{Hj} yields that, for any $\eta>0$, for all $i\in K_j =\{k_1,\ldots,k_j\}$, we have $R_i(t,x)\geq \r_i - \eta$ for $\vert x \vert < \ol c t$.

 Denote $\theta = \prod_{i\in K_j} e^{\frac{\alpha_i}{\mu_i}\eta}$. Assume that $\vert x \vert < ct$. Then, $\vert x\vert < \ol c t$ and we have
\begin{align*}
\int_0^t e^{-\frac{\alpha_k}{\mu_k}R_{k}}\partial_t R_{k} &\left(\prod_{i\in K_j}e^{-\frac{\alpha_i}{\mu_i}R_i}\right)d\tau \\
 &\leq \int_0^{\frac{\vert x \vert}{\ol c}} \partial_t R_{k} d\tau
+
\int_{\frac{\vert x \vert}{\ol c}}^t e^{-\frac{\alpha_k}{\mu_k}R_{k}}\partial_t R_{k} \left(\prod_{i\in K_j}e^{-\frac{\alpha_i}{\mu_i}R_i}\right)d\tau
\\
 &\leq \int_0^{\frac{\vert x \vert}{\ol c}} \partial_t R_{k} d\tau
+
\theta \left(\prod_{i\in K_j}e^{-\frac{\alpha_i}{\mu_i}\r_i}\right)  \int_{\frac{\vert x \vert}{\ol c}}^t e^{-\frac{\alpha_k}{\mu_k}R_{k}}\partial_t R_{k} d\tau
\\
 &\leq R_{k}\left(\frac{\vert x \vert}{\ol c},x\right) 
+
\theta \left(\prod_{i\in K_j}e^{-\frac{\alpha_i}{\mu_i}\r_i}\right)  \int_{0}^t e^{-\frac{\alpha_k}{\mu_k}R_{k}}\partial_t R_{k} d\tau.
\end{align*}

Therefore,
$$
\partial_t R_k \leq d_k \Delta R_k + \mu_kI_k^0 -\mu_kR_k + \mu_k \theta S_j (1 - e^{-\frac{\alpha_k}{\mu_k}R_{k}}) + \alpha_k S_0 R_{k}\left(\frac{\vert x \vert}{\ol c},x\right) .
$$
Now, using again Proposition \ref{prop exp}, for $\e>0$ small enough, we have
$
R_{k}(t,x) \leq  B^\e_{k,j} e^{-\frac{(c_{k}(S_{j})+\e)}{2d_{k}}(\vert x\vert - (c_{k}(S_j)+\e)t)}.
$
Up to taking $\e$ smaller if needed so that $\ol c > c_{k}(S_j) +\e$, this implies that
$$
R_{k}\left(\frac{\vert x \vert}{\ol c},x\right) \leq B^\e_{k,j} e^{-\frac{(c_{k}( S_{j})+\e)}{2d_{k}\ol c}(\ol c - (c_{k}(S_j)+\e))\vert x\vert}.
$$
In the end, taking $q:=\frac{(c_{k}( S_{j})+\e)}{2d_{k}\ol c}(\ol c - (c_{k}(S_j)+\e))$ and $C := \alpha_k S_0 B_{k,j}^\e$, we indeed find~\eqref{R supersol}.
\end{proof}
We are now in position to prove Proposition \ref{prop under}.
\begin{proof}[Proof of Proposition \ref{prop under}.]

We assume that $(H_j)$ holds true. Let us prove that $(H_{j+1})$ is verified, that is, we want to prove that, for all $c>c_{j+1}$,
\begin{equation}\label{prop sup r}
\sup_{ct<\vert x \vert}\vert R_{k_{j+1}}(t,x)\vert \underset{t\to+\infty}{\longrightarrow} 0,
\end{equation}
and that, for all $c\in (0,c_{j+1})$,
\begin{equation}\label{prop r}
\sup_{\delta < \vert x \vert < ct}\vert R_{k_{j+1}}(t,x)-\r_{j+1}\vert \underset{\delta,t\to+\infty}{\longrightarrow} 0.
\end{equation}

  \medskip
\emph{Step 1.Proof of \eqref{prop sup r}.}\\ 
Owing to Proposition \ref{prop exp}, for $\e>0$ small enough, there is $B>0$ such that
$$
R_{k_{j+1}}(t,x) \leq B e^{-\frac{(c_{k_{j+1}}(S_j)+\e)}{2d_{k_{j+1}}}\big(\vert x \vert - (c_{k_{j+1}}(S_j)+\e)t\big)},\quad \text{for all } \ t>0,\ x \in \R.
$$
Then, for any $c>c_{j+1} +\e= c_{k_{j+1}}(S_j)+\e$, we have
$$
\sup_{\vert x \vert >c t} R_{k_{j+1}}(t,x) \leq B e^{-\frac{(c_{j+1}+\e)}{2d_k}(c-c_{j+1}-\e)t} \underset{t\to +\infty}{\longrightarrow} 0.
$$
Then, because we can take $\e$ as close to zero as we want, \eqref{prop sup r} holds true.

\medskip
\emph{Step 2.Upper bound.}\\  
The remaining steps are dedicated to prove \eqref{prop r}. This particular step is dedicated to prove that, for any $\eta>0$, there are $M,\e>0$ such that
\begin{equation}\label{borne sup r}
R_{k_{j+1}}(t,x)\leq \r_{j+1} +\eta + Me^{-\e\vert x \vert}.
\end{equation}
This will imply in particular that, for any $c>0$,
$$
\limsup_{\substack{t\to +\infty \\ \delta\to+\infty}}\sup_{\delta < \vert x \vert < ct} \left(R_{k_{j+1}}(t,x)-\r_{j+1}\right)\leq 0.
$$
The remaining step below will prove that this is also true with interverting $\rho_{j+1}$ and $R_{k_{j+1}}$ (but then we will have to restrict to $c<c_{j+1}$).

First, let us take $c\in (c_{j+1},c_j)$. Observe that, owing to \eqref{R supersol} from Lemma \ref{lem sub}, we have, for any $\theta>1$, for $\vert x \vert < ct$,
$$
    \partial_t R_{k_{j+1}} \leq d_{k_{j+1}} \Delta R_{k_{j+1}}+ \mu_{k_{j+1}}I_{k_{j+1}}^0 +C e^{-q\vert x \vert} + f_{k_{j+1}}(R_{k_{j+1}},\theta S_j),
$$
for some $C,q>0$.

We denote $\r(\theta)$ the unique positive solution of the equation $f_{k_{j+1}}(\rho(\theta),\theta S_j) = 0$ (this solution exists because $z\mapsto f_{k_{j+1}}(z,\theta S_j)$ is concave and vanishes at $z=0$ where it has a strictly positive derivative, and goes to $-\infty$ as $z\to +\infty$). We have $\r(\theta)\to \r_{j+1}$ as $\theta \to 1$.

Let us define the function
$$
u(t,x) = \r(\theta) + Me^{-\e x}
$$
Because the function $f_{k_{j+1}}(\cdot,\theta S_j)$ is concave, we have $f_{k_{j+1}}(u,\theta S_j) \leq f_{k_{j+1}}^\prime(\r(\theta),\theta S_j)M e^{-\e x }$. Therefore,
    $$
    \partial_t u -d_{k_{j+1}} \Delta u -f_{k_{j+1}}(u,\theta S_j)\geq Me^{-\e x }(-d_{k_{j+1}}\e^2-  f_{k_{j+1}}^\prime(\r(\theta),\theta S_j) ).
    $$
    Up to taking $M>0$ large enough and $\e>0$ small enough, we can guarantee that
    $$
     M (-d_{k_{j+1}}\e^2-  f_{k_{j+1}}^\prime(\r(\theta),\theta S_j) )\geq (\mu_{k_{j+1}}I_{k_{j+1}}^0 + C e^{-q\vert x \vert})e^{\e x},
    $$
    indeed, $ f_{k_{j+1}}^\prime(\r(\theta),\theta S_j)<0$ and $I_{k_{j+1}}^0$ is compactly supported.
    
Therefore, $u$ and $ R_{k_{j+1}}$ are respectively supersolution and subsolution of the same parabolic equation on the domain $\vert x \vert < ct$ ($u$ is actually supersolution everywhere in space and time). We have $u(0,\cdot)\geq  R_{k_{j+1}}(0,\cdot)=0$. To apply the parabolic comparison principle, we need to have $ R_{k_{j+1}}(t,\pm ct)\leq u(t,ct)$, but this comes again from the exponential estimates of Proposition \ref{prop exp}, up to taking $M$ large enough (and because we took $c>c_{j+1}$). 

Doing the same reasoning with the supersolution $\rho(\theta)+Me^{\e x}$, we finally end up with finding, thanks to the parabolic comparison principle, that
$$
R_{k_{j+1}}(t,x)\leq \rho(\theta) + Me^{-\e \vert x \vert},\quad t>0,\ \vert x \vert <ct.
$$
Taking $\theta$ close enough to $1$ yields that \eqref{borne sup r} is verified at least for $\vert x\vert <ct$, but because $R_{k_{j+1}}$ is non-decreasing with respect to $t$, this is actually true for all $t>0, x\in \R$.

  \medskip
\emph{Step 3. Lower bound.}\\
It now remains to prove that, for $c\in (0,c_{j+1})$, we have
$$
\limsup_{\substack{\delta \to+\infty\\t\to +\infty}}\sup_{\delta < \vert x \vert < ct} \left(\r_{j+1}-R_{k_{j+1}}(t,x)\right)\leq 0,
$$
or, equivalently, that
\begin{equation}\label{equi prop}
\liminf_{\substack{\delta \to+\infty\\t\to +\infty}}\inf_{\delta < \vert x \vert < ct} R_{k_{j+1}}(t,x) \geq \r_{j+1}.
\end{equation}
To do so, we use \eqref{eq mf} from Lemma \ref{lem sub}. Let us take $\ul c, \ol c \in (\max\{\t c , c\} , c_{j+1})$ (where $\t c$ is from Lemma \ref{lem sub}), such that $\ul c <\ol c$.

Let us take $\eta>0$ small enough and $\omega \in (0,1)$ close enough to $1$ such that
$$
2\sqrt{d_{k_{j+1}}(\omega \alpha_{k_{j+1}}S_j - \mu_{k_{j+1}}-\eta)} > \ol c,
$$
which is possible because $2\sqrt{d_{k_{j+1}}( \alpha_{k_{j+1}}S_j - \mu_{k_{j+1}})} = c_{j+1}> \ol c$.
The equation \eqref{eq mf} from Lemma \ref{lem sub} then implies that there is $T>0$ such that
$$
 \partial_t R_{k_{j+1}} \geq d_{k_{j+1}} \Delta R_{k_{j+1}} +f_{k_{j+1}}(R_{k_{j+1}},\omega S_j) - \eta R_{k_{j+1}},\quad \text{for}\ t>T,\ \vert x \vert > \ul c t.
$$
Let us now define $u(t,x)$ as the solution of the equation
$$
 \partial_t u = d_{k_{j+1}} \Delta u +f_{k_{j+1}}(u,\omega S_j) - \eta u,\quad \text{for}\ t>T,\ \vert x \vert > \ul c t,
$$
with (moving) Dirichlet boundary condition $u(t,\pm \ul c t)=0$ for $t>T$, and with initial condition at time $t=T$ given by $u(T,\cdot)= R_{k_{j+1}}(T,\cdot)$.

Then, applying the parabolic comparison principle on the domain $E:= \{(t,x) \ : \ t>T,\ \vert x \vert >\ul c t\}$, we find that $R_{k_{j+1}}(t,x)\geq u(t,x)$ for $(t,x) \in E$.

Now, because $R_{k_{j+1}}(t,x)$ is non-decreasing with respect to $t$, and by continuity, we have that there is $x^\delta_t$ such that $\delta<\vert x^\delta_t\vert < ct$ that satisfies
$$
\inf_{\delta < \vert x \vert <ct}R_{k_{j+1}}(t,x) = R_{k_{j+1}}(t,x^\delta_t) \geq R_{k_{j+1}}\left(\frac{\vert x^\delta_t\vert}{\ol c},x^\delta_t\right) .
$$
Observe that we used the fact that $\ol c >c$ here.

Now, for $\delta,t$ large enough, we have $\left(\frac{\vert x^\delta_t\vert}{\ol c},x^\delta_t\right)\in E$. Hence, for $\delta,t$ large, we have
$$
\inf_{\delta < \vert x \vert <ct}R_{k_{j+1}}(t,x)\geq u\left(\frac{\vert x^\delta_t\vert}{\ol c},x^\delta_t\right).
$$
Now, owing to classical results from reaction-diffusion equations theory, we have
$$
u(t,\pm \ol c t) \to \r_{j+1}(\omega,\eta),
$$
where $\r_{j+1}(\omega,\eta)$ is the unique positive solution of
$$
f_{k_{j+1}}(\r_{j+1}(\omega,\eta),\omega S_{j})-\eta \r_{j+1}(\omega,\eta)=0.
$$
This solution exists and is unique provided $\eta,\omega$ are small enough and close enough to $1$ respectively. Clearly, we have $\r_{j+1}(\omega,\eta)\to \r_{j+1}$ as $\eta$ goes to zero and $\omega$ goes to $1$, which yields \eqref{equi prop}, and concludes the proof.
\end{proof}

To conclude the proof of Theorem \ref{main th}, we need the following:
\begin{prop}\label{prop no prop}
Let $(S,I_1,\ldots, I_N,R_1,\ldots,R_N)$ be the solution of \eqref{syst} arising from the initial datum $(S_0,I_1^0,\ldots, I_N^0,0,\ldots,0)$.

Assume that $(H_p)$ holds true. Then, for every $k\neq k_1,\ldots,k_p$, the trait $k$ vanishes.
\end{prop}
\begin{proof}
Let $k\neq k_1,\ldots,k_p$ be chosen. Using \eqref{R supersol} from Lemma \ref{lem sub}, we know that, for $\theta>1$, for $c<c_p$, there are $C,q,T>0$ such that
$$
    \partial_t R_k \leq d_k \Delta R_k + \mu_kI_k^0 + C e^{-q\vert x \vert} + f_{k}(R_k,\theta S_p),\quad \text{for}\ t>T,\ \vert x \vert < c t.
$$
By definition of the propagation sequence, up to taking $\theta$ close enough to $1$ and thanks to hypothesis \eqref{hyp 2}, we can guarantee that $-\eta := \alpha_k \theta S_p - \mu_k<0$. Hence, by concavity of $z\mapsto f_k(z,\theta S_p)$, and because $f^\prime_k(0,\theta S_p) = -\eta$, we have
$$
    \partial_t R_k \leq d_k \Delta R_k + \mu_kI_k^0 + C e^{-q\vert x \vert} -\eta R_k,\quad \text{for}\ t>T,\ \vert x \vert < c t.
$$
Now, owing to the estimates from Proposition \ref{prop exp}, we know that, for every $\e>0$, we have (remember that $c_k(S_p)=0$)
$$
R_k(t,x)\leq B^\e_{k,p} e^{-\frac{\e}{2d_k}\big(\vert x \vert - \e t\big)},\quad \text{for all } \ t>0,\ x \in \R.
$$
Let now $M,\lambda>0$ be fixed, and define $w(t,x) := Me^{-\lambda x}$. We have
$$
\partial_t w -d_k\Delta w +\eta w = (-d_k \lambda^2 + \eta)Me^{-\lambda x}\geq \mu_k I_k^0 + C e^{-q\vert x\vert},
$$
provided $\lambda$ is small enough and $M$ is large enough. Moreover, up to decreasing $\lambda$ if needed, we have
$$
M e^{-\lambda c t}\geq B^\e_{k,p} e^{-\frac{\e}{2d_k}\big(c - \e \big)t}.
$$
Therefore, $R_k(t,\pm c t)\leq w(t,\pm c t)$. Hence, we can apply the parabolic comparison principle to get that
$$
R_k(t,x)\leq Me^{-\lambda x}, \quad  \text{for }\ t>T,\ \vert x\vert <ct.
$$
Because $R_k(t,x)$ is non-decreasing with respect to $t$, then this is actually true for all $t>0,x \in \R$. Therefore, the $k$-th trait does not propagate, hence the result.
\end{proof}

\section{Qualitative properties}\label{sec quali}

Now that we have proved Theorem \ref{main th}, we use it to study some qualitative properties of our model \eqref{syst}. More precisely, Theorem \ref{main th} tells us that the dynamic of the system is given by the propagation sequences from  Definition \ref{def prop seq}. Hence, understanding how the propagation sequences depend on the parameters of the model will yield qualitative properties.

As we already mentioned, this task may be rather intricate, because the propagation sequences depend on the parameters in a complex way.
Hence, we do not plan on giving a general picture here but rather we want to emphasize the three properties mentioned at the end of Section \ref{sec res}. These three properties are stated as the three following results, that are corollaries of Theorem \ref{main th}.\\

First, we show that the basic reproduction numbers, introduced in Section \ref{sec rap} as the ratio $\frac{\alpha S_0}{\mu}$, are not good indicators of the dynamic of the system.

\begin{corollary}\label{cor 1}
Consider the system \eqref{syst} with $N=2$. Take $d_1=d_2 =1$. Then, one can choose $\alpha_1,\alpha_2,\mu_1,\mu_2$ and $S_0$ such that
$$
1<\frac{\alpha_1 S_0}{\mu_1} < \frac{\alpha_2 S_0}{\mu_2},
$$
but where the trait $1$ propagates and the trait $2$ vanishes.
\end{corollary}
This corollary gives the existence of situations where the trait with the lower basic reproduction number propagates, and where the trait with the larger basic reproduction number disappears. We mention it can be adapted to the case where $N\geq 2$.

This is in contradiction with the intuition one could have built from the model with one single trait.\\

    The second property we prove concerns the contagiousness of each successive trait. More precisely, we show that each trait that propagates is less contagious, and has a lower recovery rate, than the traits that have propagated before. This is not completely direct from Theorem \ref{main th}, which only tells us that each trait that propagates has a smaller speed than the trait that propagates before.
\begin{corollary}\label{cor 2}
Consider the system \eqref{syst} with all diffusion equals, i.e., $d_1=\ldots=d_N$. Let $(k_i)_{i\in\llbracket 1,p\rrbracket}$ be the sequence of the indices that propagate, given in Definition \ref{def prop seq}.

Assume that the coefficients satisfy the hypotheses of Theorem \ref{main th} (in particular \eqref{hyp} and \eqref{hyp 2}).

Then, for all $i\in\llbracket1 , p-1 \rrbracket$, one has
$$
\alpha_{k_{i+1}}< \alpha_{k_i} \quad \text{ and }\quad \mu_{k_{i+1}}<\mu_{k_i}.
$$
\end{corollary}
To conclude, we present a result concerning the total number of casualties, or, what is equivalent, the final number of susceptible individuals.

We denote 
$$
S_{\infty}:=\lim_{\vert x \vert \to +\infty}\lim_{t\to+\infty}S(t,x).
$$
As we already mention in Section \ref{sec res}, when we can apply Theorem \ref{main th}, we have $S_\infty = S_p$, where $S_p$ is given by the propagation sequences.

When $N=1$, and when $\frac{\alpha S_0}{\mu}\leq 1$, the single trait does not propagate, and $S_\infty = S_0$, whereas when $\frac{\alpha S_0}{\mu}>1$, the trait propagates, and $S_\infty = S_0 e^{-\frac{\alpha}{\mu}\rho}$, were $\rho$ solves $S_0(1-e^{-\frac{\alpha}{\mu}\rho})=\rho$.

We can rewrite this by saying that, in this case $N=1$, $S_\infty$ is the smallest solution of
$$
S_0 - \frac{\mu}{\alpha}\ln(S_0) = S_\infty - \frac{\mu}{\alpha}\ln(S_\infty).
$$
This comes easily from the result we just mentioned. The function $f(z) = z - \frac{\mu}{\alpha}\ln(z)$ is strictly convex, goes to $+\infty$ as $z$ goes to $0$ or to $+\infty$, and has a unique minimum, reached at $z = \frac{\mu}{\alpha}$.

Observe that we necessarily have $S_0e^{-\frac{\alpha}{\mu}\rho}<\frac{\mu}{\alpha}$ when $S_0>\frac{\mu}{\alpha}$ (this is natural, when a disease propagates, the number of susceptible left untouched is too small for disease to propagate again). We shall use this fact several time in the following proofs.

We can plot $S_\infty$ as a function of $S_0$. One would obtain the following graph:
\begin{center}
    \includegraphics[scale=0.7]{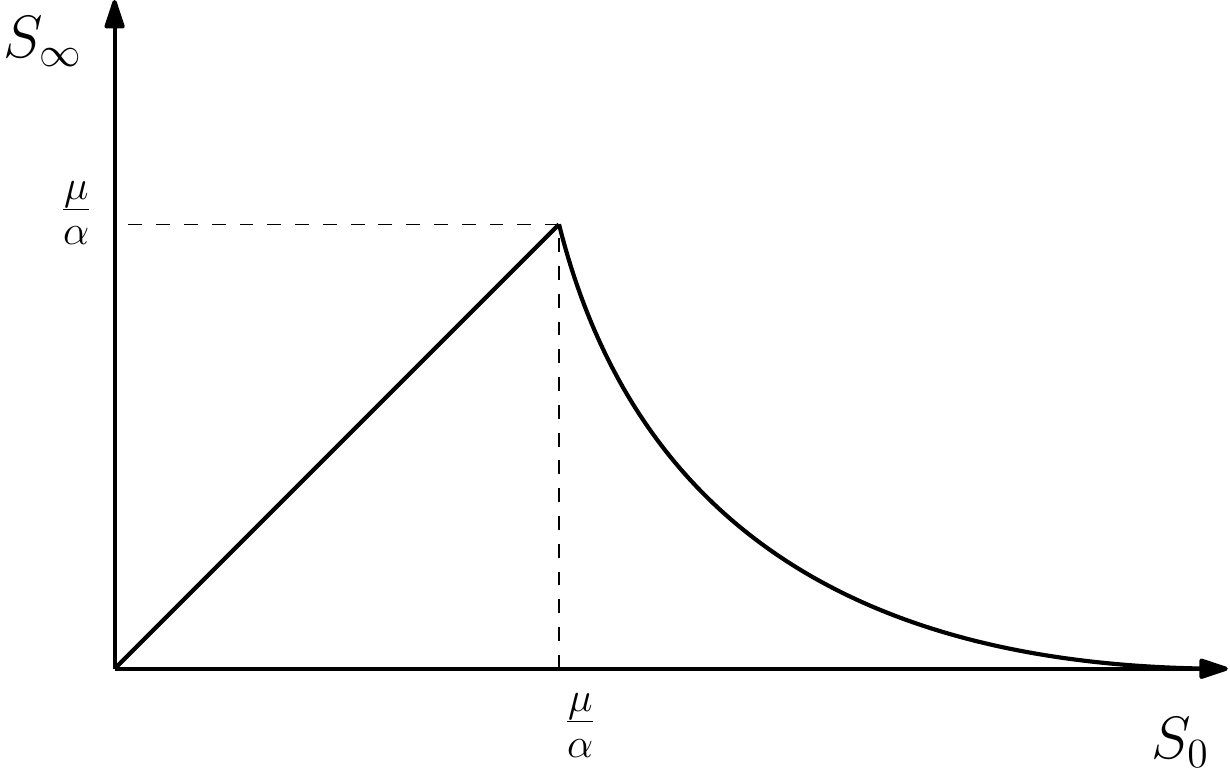}
\end{center}
Our next corollary shows how the final number of susceptible $S_\infty$ depends on the initial number of susceptible $S_0$ when $N=2$.

\begin{corollary}\label{cor 3}
Consider the system \eqref{syst} with $N=2$. We denote $S_\infty := \lim_{\vert x \vert \to +\infty}\lim_{t\to +\infty}S(t,x)$.

Without loss of generality, assume that
$$
\frac{\mu_1}{\alpha_1} < \frac{\mu_2}{\alpha_2}.
$$

Then, there are $0<\ul r \leq \ol r$ such that
\begin{itemize}
    \item If 
    $$
    \frac{d_2}{d_1}<\ul r,
    $$
    the situation is the same as if there were only the trait $1$ in the model:
    \begin{itemize}
        \item When $S_0 < \frac{\mu_1}{\alpha_1}$, we have $S_\infty = S_0$.
        \item  When $S_0>\frac{\mu_1}{\alpha_1}$, only the trait $1$ propagates, and $S_\infty$ is the smallest solution of
    $$
    S_\infty -\frac{\mu_1}{\alpha_1}S_\infty = S_0 - \frac{\mu_1}{\alpha_1}\ln(S_0).
    $$
    \end{itemize}

    \item If 
    $$
    \frac{d_2}{d_1}> \ol r
    $$
    there are $\ul S, \ol S$ that verify $\frac{\mu_1}{\alpha_1}<\frac{\mu_2}{\alpha_2}<\ul S< \ol S$
    and $\e\geq 0$ such that
    \begin{itemize}
        \item When $S_0<\frac{\mu_1}{\alpha_1}$, no trait propagate, and $S_\infty = S_0$.
        \item When $S_0\in (\frac{\mu_1}{\alpha_1}, \ul S)$, only the trait $1$ propagates, and then $S_\infty = S_0e^{-\frac{\alpha_1}{\mu_1}\rho_1(S_0)}$.
        \item When $S_0 \in (\ul S, \ul S +\e)$, we can not apply Theorem \ref{main th}.
        \item When $S_0 \in (\ul S+\e, \ol S)$, then the traits $2$ and $1$ propagate, and $S_\infty = S_1 e^{-\frac{\alpha_1}{\mu_1}\rho_1(S_1)}$, where $S_1 = S_0 e^{-\frac{\alpha_2}{\mu_2}\rho_2(S_0)}$.
        \item When $S_0>\ol S$, then only the trait $2$ propagates, and $S_\infty = S_0 e^{-\frac{\alpha_2}{\mu_2}\rho_2(S_0)}$.
    \end{itemize}
    In addition, we have that $\ul S \to \frac{\mu_2}{\alpha_2}$ and $\e \to 0$ as $\frac{d_2}{d_1} \to +\infty$, and $\ol S$ does not depend on $d_1,d_2$.
\end{itemize}
\end{corollary}
Let us give here some remarks concerning this lemma. Of course, the most interesting situation is when $d_1<<d_2$. In this case, plotting $S_\infty$ as a function of $S_0$ would give
\begin{center}
    \includegraphics[scale=0.7]{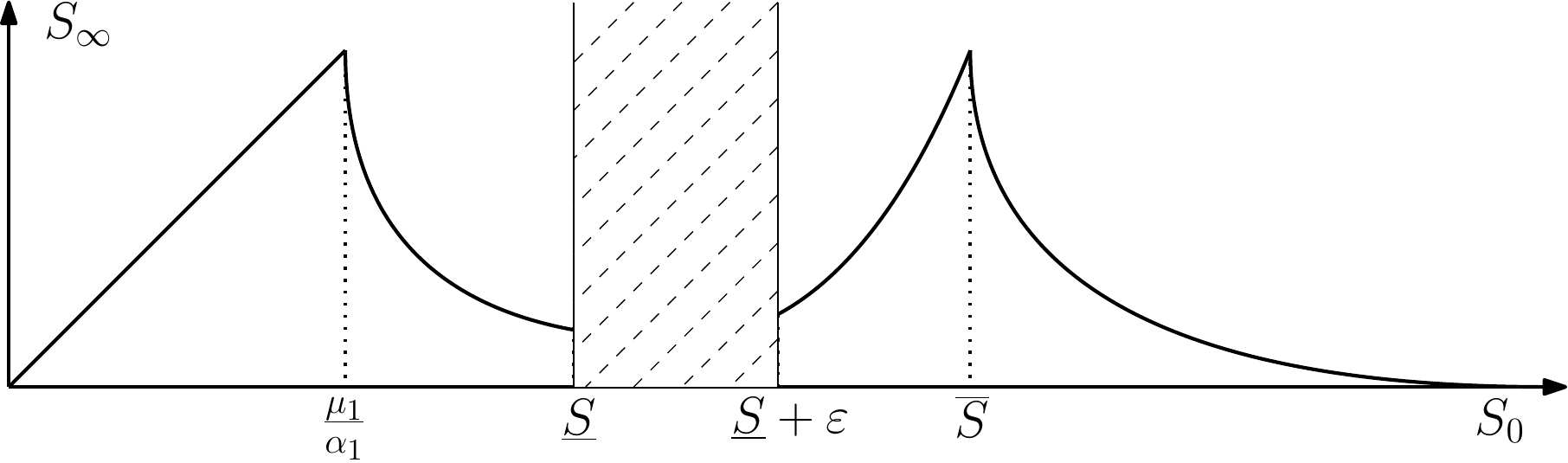}
\end{center}
The dashed zone, for $S_0\in (\ul S, \ul S+\e)$, corresponds to the zone where we can not apply our theorem (\eqref{hyp} is not verified there). However, observe that the zone can be made as small as we want up to increasing $d_2$ or decreasing $d_1$.

We now turn to the proofs of these results.
\begin{proof}[Proof of Corollary \ref{cor 1}]
We take $d_1 = d_2 =1$ fixed without loss of generality.

Let $\alpha_1,\mu_1,S_0$ fixed in such a way that
$$
\frac{\alpha_1 S_0}{\mu_1} >1.
$$
Now, let $S_1 = S_0 e^{-\frac{\alpha_1}{\mu_1}\rho_1}$ where $\rho_1$ is the unique positive solution of $f_1(\rho_1,S_0)=0$. We have
$$
\frac{\alpha_1 S_1}{\mu_1}<1.
$$
Take $\lambda \in (\frac{\alpha_1}{\mu_1},\frac{1}{S_1})$ and define $\alpha_2 = \e \lambda$ and $\mu_2 = \e$, for $\e>0$. Then
$$
\frac{\alpha_1 S_0}{\mu_1} < \lambda S_0 = \frac{\alpha_2 S_0}{\mu_1}.
$$
Moreover, up to taking $\e>0$ small enough, we ensure that
$$
\alpha_2 S_0 - \mu_2 = \e( \lambda S_0 - 1) <\alpha_1 S_0 - \mu_1.
$$
We also have
$$
\alpha_2 S_1 - \mu_2 = \e( \lambda S_1 - 1)<0.
$$
We can compute the propagation sequence associated with these parameters: we find $p=1$ and $k_1 = 1$. Moreover, the hypotheses of Theorem \ref{main th} are verified, then only the trait $1$ propagates and the trait $2$ goes extinct.
\end{proof}

We now turn to the proof of the second qualitative property.
\begin{proof}[Proof of Corollary \ref{cor 2}]

Assume that the coefficients satisfy the hypotheses. Then, by definition of the propagation sequences, we know that, for all $i\in \llbracket 1,p-1\rrbracket$,
$$
c_{k_{i+1}}(S_{i-1}) < c_{k_i}(S_{i-1}),
$$
that is (recall that we assume that the diffusions are equal),
$$
\alpha_{k_{i+1}}S_{i-1} - \mu_{k_{i+1}} <\alpha_{k_i}S_{i-1} - \mu_{k_i},
$$
hence
\begin{equation}\label{mu a}
(\alpha_{k_{i+1}}-\alpha_{k_{i}})S_{i-1}<\mu_{k_{i+1}}- \mu_{k_{i}}.
\end{equation}
In addition, because the trait $k_{i+1}$ propagates, we necessarily have
\begin{equation}\label{i1}
\alpha_{k_{i+1}}S_i - \mu_{k_{i+1}} >0.
\end{equation}
Moreover, by definition of $S_i$, we have  
\begin{equation}\label{i2}
\alpha_{k_i}S_i - \mu_{k_i}<0.
\end{equation}
Combining \eqref{mu a}, \eqref{i1} and \eqref{i2} yields
$$
(\alpha_{k_{i+1}}-\alpha_{k_{i}})S_{i-1}<(\alpha_{k_{i+1}}-\alpha_{k_{i}})S_{i}.
$$
Now, because $S_{i-1}>S_{i}$, we have
$$
\alpha_{k_{i+1}}<\alpha_{k_{i}}.
$$
Combining this again with \eqref{i1} and \eqref{i2} yields
and this implies, thanks to \eqref{mu a}
$$
\mu_{k_{i+1}}<\mu_{k_i},
$$
and this concludes the proof.
\end{proof}

We now prove our third qualitative property.
\begin{proof}[Proof of Corollary \ref{cor 3}]
We define 
$$
\ul r := \frac{\alpha_1}{\alpha_2}.
$$
Then, when $\frac{d_2}{d_1}\leq \ul r$, we have $d_2\alpha_2 \leq d_1 \alpha_1$ and, for all $S_0>0$,
$$
c_1(S_0)\geq c_2(S_0).
$$
If $S_0<\frac{\mu_2}{\alpha_2}$, then $\alpha_2 S_0 - \mu_2<\alpha_1 S_0 - \mu_1<0$. We can compute the propagation sequences: we find $p=0$, the sequences are empty. We can apply Theorem \ref{main th} to get that no trait propagate, and then $S_\infty = S_0$.

If $S_0>\frac{\mu_2}{\alpha_2}$, then $c_1(S_0)>0$. Computing the propagation sequence, we find that $k_1 = 1$, $S_1 = S_0e^{-\frac{\alpha_1}{\mu_1}\rho_1(S_0)} < \frac{\mu_1}{\alpha_1}$. This implies that $S_1 < \frac{\mu_2}{\alpha_2}$, then $\alpha_2 S_1 - \mu_2<0$, hence $p=1$. We can apply Theorem \ref{main th} to find that only the trait $1$ propagates, and $S_\infty = S_1$.\\

Now, we let $\ol r>0$ be a real number to be taken sufficiently large after (unlike $\ul r$, it will not be explicit). We assume for now that $\ol r > \frac{\alpha_1}{\alpha_2}$.

If $S_0<\frac{\mu_1}{\alpha_1}$, the situation is the same as above, no trait propagate and $S_\infty = S_0$.

Now, let
$$
\ul S := \frac{d_2 \mu_2 - d_1 \mu_1}{d_2 \alpha_2 - d_1 \alpha_1}.
$$
Because we assumed that $\ol r >\frac{\alpha_1}{\alpha_2}$, an easy computation shows that $\ul S >\frac{\mu_2}{\alpha_2}$, and, for $S_0 \in (\frac{\mu_1}{\alpha_1},\ul S)$, we have $c_1(S_0)>c_2(S_0)$.

Hence, computing the propagation sequence, we find that $k_1 =1$ and $S_1 = S_0e^{-\frac{\alpha_1}{\mu_1}\rho_1(S_0)}$. We have $S_1<\frac{\mu_1}{\alpha_1}<\frac{\mu_2}{\alpha_2}$, then $d_2(\alpha_2 S_1 - \mu_2)<0$, and we find $p=1$.

We can apply Theorem \ref{main th}, we find that only the trait $1$ propagates and $S_\infty = S_1 = S_0e^{-\frac{\alpha_1}{\mu_1}\rho_1(S_0)}$.

Now, let $\ol S$ be the largest solution of
$$
\ol S - \frac{\mu_2}{\alpha_2}\ln(\ol S) = \frac{\mu_1}{\alpha_1} - \frac{\mu_2}{\alpha_2}\ln(\frac{\mu_1}{\alpha_1}).
$$
On the one hand, $\ol S$ does not depend on $d_1,d_2$. On the other hand, it is readily seen that $\ul S$ is a decreasing function of $\frac{d_2}{d_1}$ and that, when $\frac{d_2}{d_1}\to+\infty$, we have $\ul S \to \frac{\mu_2}{\alpha_2}$. We assume from now on that $\ol r$ is large enough so that, for $\frac{d_2}{d_1}>\ol r$, we have $\ol S > \ul S$.

We find that, for $S_0>\ul S$, we have
$$
c_1(S_0)<c_2(S_0).
$$
Computing the propagation sequence, we find $k_1 = 2$, and $S_1 = S_0e^{-\frac{\alpha_2}{\mu_2}\rho_2(S_0)}$. If $S_0\in (\ul S, \ol S)$, owing to our choice of $\ol S$, we have $S_1 >\frac{\mu_1}{\alpha_1}$, and then
$$
c_1(S_1)>0.
$$
Then, we find $p=2$, and $S_\infty = S_2 = S_1 e^{-\frac{\alpha_1}{\mu_1}\rho_1(S_1)}$.

Now, to apply Theorem \ref{main th}, we need to verify that \eqref{hyp} holds true, i.e., we can apply the theorem only when $S_0$ is such that
$$
c_1(S_0)+c_1(S_1)<c_2(S_0).
$$
Unfortunately, this can not be everywhere verified (indeed, $c_2(S_0) - c_1(S_0) \to 0$ when $S_0 \to \ul S$). Finding the minimal $S_0$ that satisfies this is rather intricate ($S_1$ is a function of $S_0$), but observe that
$$
c_1(S_0)+c_1(S_1)< 2c_1(S_0) = 4\sqrt{d_1(\alpha_1 S_0 - \mu_1)}.
$$
Up to increasing $\ol r$ so that $\ol r> 4\frac{\alpha_1}{\alpha_2}$, it is easy to see that
$$
4\sqrt{d_1(\alpha_1 S_0 - \mu_1)} < 2\sqrt{d_2(\alpha_2 S_0 - \mu_2)} \iff S_0 > \ul S +\e,
$$
where $\e = \frac{3 d_2d_1(\mu_2 \alpha_1 - \mu_1 \alpha_2)}{(d_2\alpha_2 -4d_1\alpha_1)(d_2\alpha_2 - d_1 \alpha_1)}$ indeed goes to zero when $\frac{d_2}{d_1}\to+\infty$.

Now, when $S_0>\ol S$, we still have $c_2(S_0)>c_1(S_0)$, but now $\alpha_1S_1 - \mu_1<0$. Hence, the propagation sequence contains one element, $p=1$, and we can apply Theorem \ref{main th} to find that only the trait $2$ propagates, and $S_\infty = S_1 = S_0e^{-\frac{\alpha_2}{\mu_2}\rho_2(S_0)}$.
\end{proof}

\end{document}